\documentclass[a4paper,11pt]{article}

\usepackage{amsmath,amsthm,amssymb,comment}
\usepackage{graphicx,subcaption,tikz}							                              
\usepackage[shortlabels]{enumitem}
\usepackage{hyperref}       
\hypersetup{
    colorlinks=true,
    citecolor = blue,
    linkcolor=black,
    filecolor=magenta,
    urlcolor=cyan,
    bookmarks=true,
}
\usepackage[margin=1in]{geometry}
\usepackage{todonotes}

\newtheorem{theorem}{Theorem}
\newtheorem{lemma}{Lemma}
\newtheorem{claim}{Claim}
\newtheorem{conjecture}{Conjecture}

\setenumerate[1]{label=(\thesection.\arabic*)}
\newcommand\emitem[1]{\item{ #1}}		
\newcommand \e {\hfill {\tiny $\blacksquare$}}
\newcommand{\case}[1]{\noindent {\bf Case #1.}}

\title{{\bf $(2P_2,K_4)$-Free Graphs are 4-Colorable}}

\author
{
Serge Gaspers\thanks{School of Computer Science and Engineering, UNSW Sydney, Sydney 2052, Australia.}
\thanks{Decision Sciences, Data61, CSIRO, Sydney 2052, Australia.}
\and
Shenwei Huang\thanks{Department of Mathematics, Wilfrid Laurier University, Waterloo N2L3C5, Canada.} 
 }

\date{August 10, 2018}

\begin{document}

\maketitle

\begin{abstract}
In this paper, we show that every $(2P_2,K_4)$-free graph is 4-colorable. 
The bound is attained by the five-wheel and the complement
of the seven-cycle. This answers an open question by Wagon \cite{Wa80} in the 1980s.
Our result can also be viewed as a result in the study of  the Vizing bound for graph classes.
A major open problem in the study of computational complexity of graph coloring is whether 
coloring can be solved in polynomial time for $(4P_1,C_4)$-free graphs. 
Lozin and Malyshev \cite{LM17} conjecture that the answer is yes.
As an application of our main result, 
we provide the first positive evidence to the conjecture by giving a 2-approximation algorithm for coloring $(4P_1,C_4)$-free graphs.
\end{abstract}



\section{Introduction}

All graphs in this paper are finite and simple.
We say that a graph $G$ {\em contains} a graph $H$ if $H$ is
isomorphic to an induced subgraph of $G$.  A graph $G$ is
{\em $H$-free} if it does not contain $H$. 
For a family of graphs $\mathcal{H}$,
$G$ is {\em $\mathcal{H}$-free} if $G$ is $H$-free for every $H\in \mathcal{H}$.
In case $\mathcal{H}$ consists of two graphs, we write
$(H_1,H_2)$-free instead of $\{H_1,H_2\}$-free.
As usual, let $P_n$ and $C_n$ denote
the path and the cycle on $n$ vertices, respectively. The complete
graph on $n$ vertices is denoted by $K_n$. The {\em $n$-wheel} $W_n$ is the graph obtained
from $C_n$ by adding a new vertex and making it adjacent to every vertex in $C_n$.
For two graphs $G$ and $H$, we use $G+H$ to denote the \emph{disjoint union} of $G$ and $H$.
For a positive integer $r$, we use $rG$ to denote the disjoint union of $r$ copies of $G$.
The \emph{complement} of $G$ is denoted by $\overline{G}$.
A {\em hole} in a graph is an induced cycle of length at least 4.
A hole is {\em odd} if it is of odd length.

A \emph{$q$-coloring} of a graph $G$ is a function $\phi:V(G)\longrightarrow \{ 1, \ldots ,q\}$ such that
$\phi(u)\neq \phi(v)$ whenever $u$ and $v$ are adjacent in $G$. We say that $G$ is {\em $q$-colorable} if $G$ admits a $q$-coloring.
The \emph{chromatic number} of  $G$, denoted by $\chi (G)$, is the minimum number $q$ such that $G$ is $q$-colorable.
The \emph{clique number} of $G$, denoted by $\omega(G)$, is the size of a largest clique in $G$.
Obviously, $\chi(G)\ge \omega(G)$ for any graph $G$.
The {\em maximum degree} of a graph $G$ is denoted by $\Delta(G)$.

A family $\mathcal{G}$ of graphs is said to be \emph{$\chi$-bounded} if 
there exists a function $f$ such that for every graph
$G\in \mathcal{G}$ and every induced subgraph $H$ of $G$ it holds that
$\chi(H)\le f(\omega(H))$. The function $f$ is called a \emph{$\chi$-binding} function
for $\mathcal{G}$. The class of perfect graphs (a graph $G$ is \emph{perfect} if for every induced subgraph $H$ of $G$ 
it holds that $\chi(H)=\omega(H)$), for instance, is a  $\chi$-bounded family with $\chi$-binding
function $f(x)=x$. 
Therefore, $\chi$-boundedness is a generalization of perfection.
The notion of $\chi$-bounded families was introduced by Gy{\'a}rf{\'a}s \cite{Gy87} who make the following conjecture.

\begin{conjecture}[Gy{\'a}rf{\'a}s \cite{Gy73}]
For every forest $T$, the class of $T$-free graphs is $\chi$-bounded.
\end{conjecture}

Gy{\'a}rf{\'a}s \cite{Gy87} proved the conjecture for $T=P_t$: every $P_t$-free graph
$G$ has $\chi(G)\le (t-1)^{\omega(G)-1}$.  The result was slightly improved by Gravier, Ho{\`a}ng and Maffray
in \cite{GHM03} that every $P_t$-free graph $G$ has $\chi(G)\le (t-2)^{\omega(G)-1}$.
This implies that  every $P_5$-free graph $G$ has $\chi(G)\le 3^{\omega(G)-1}$. 
Note that this $\chi$-binding function is exponential in $\omega(G)$. 
For $\omega(G)=3$, Esperet, Lemoine, Maffray and Morel \cite{ELMM13} obtained the optimal bound on the chromatic number:
every $(P_5,K_4)$-free graph is 5-colorable. They also demonstrated a $(P_5,K_4)$-free graph whose chromatic number is 5.
On the other hand, a polynomial $\chi$-binding function for the class of $2P_2$-free graphs was shown by
Wagon \cite{Wa80} who proved that every such graph has $\chi(G)\le \binom{\omega(G)+1}{2}$. 
This implies that every $(2P_2,K_4)$-free graph is 6-colorable.  In \cite{Wa80} it was asked if there exists
a $(2P_2,K_4)$-free graph whose chromatic number is 5 or 6.
We observe that the $(P_5,K_4)$-free graph with chromatic number 5 given in \cite{ELMM13} contains an induced $2P_2$.

In this paper we settle Wagon's question \cite{Wa80} by proving the following theorem.

\begin{theorem}\label{thm:main}
Every $(2P_2,K_4)$-free graph $G$ has $\chi(G)\le 4$. 
\end{theorem}

The bound in \autoref{thm:main} is attained by  the five-wheel $W_5$ and the complement of a seven-cycle $\overline{C_7}$. 
Hence, we obtain the optimal $\chi$-bound for the class of $2P_2$-free graphs when the clique number is 3.
A family $\mathcal{G}$ of graph is said to satisfy the {\em Vizing bound} if $f(x)=x+1$ is a $\chi$-binding function for $\mathcal{G}$.
The definition was motivated by the classical Vizing's Theorem \cite{Vi64} on the chromatic index  $\chi'(G)$ of graphs which states that
$\chi'(G)\le \Delta(G)+1$ for any graph $G$. This is equivalent to say that the class of line graphs satisfies the Vizing bound.
Our result (\autoref{thm:main}) shows that the class of $(2P_2,K_4)$-free graphs also satisfies the Vizing bound.
We refer to Randerath and Schiermeyer \cite{RS04} and Fan, Xu, Ye and Yu \cite{FXYY14} for more results on the Vizing bound
for various $\mathcal{H}$-free graphs.

We also note that our proofs of \autoref{thm:main} below are algorithmic: one can easily follow the steps of the proof and give a 4-coloring
of the input graph in polynomial time.

\medskip
\noindent {\bf An application.} Let \textsc{Coloring} denoted the computational problem of determining the chromatic number of a graph.
In the past two decades, there has been an overwhelming attention on the complexity of \textsc{Coloring} $\mathcal{H}$-free graphs.
The starting point is a result due to Kr\'al', Kratochv\'{\i}l, Tuza, and Woeginger~\cite{KKTW01} who
gave a complete classification of the complexity of  \textsc{Coloring} for the case where ${\cal H}$ consists of a single graph~$H$:
if $H$ is an induced subgraph of $P_4$ or of $P_1+ P_3$, then
\textsc{Coloring} restricted to $H$-free graphs is polynomial-time solvable, otherwise it is NP-complete. 
Afterwards, researchers started to study \textsc{Coloring} restricted to $(H_1,H_2)$-free graphs.
Despite much efforts of top researchers in the area the complexity of \textsc{Coloring} are known only for some
pairs of $H_1$ and $H_2$, see \cite{GJPS17} for a summary of the known partial results.
Even solving the problem for particular pairs of $H_1$ and $H_2$ requires substantial work, see
\cite{DLRR12, Ma14, HL15, HJP15, LM17,KMP18} for instance.  
Lozin and Malyshev \cite{LM17} demonstrated
that the classification is already problematic even if both $H_1$ and $H_2$ are $4$-vertex graphs:
they determined the complexity of \textsc{Coloring} for all such pairs with three exceptions.
One of the three unknown pairs is $(4P_1,C_4)$. Lozin and Malyshev \cite{LM17} conjecture that \textsc{Coloring}
can be solved in polynomial time for $(4P_1,C_4)$-free graphs. The problem was listed as an important
open problem in the survey on the computational complexity of coloring graphs with forbidden subgraphs
by Golovach, Johnson, Paulusma and Song \cite{GJPS17}.

Here we use \autoref{thm:main} to give a 2-approximation algorithm
for coloring $(4P_1,C_4)$-free graphs. This is the first general result towards a polynomial-time algorithm for the problem,
although Fraser, Hamel, Ho\`ang, Holmes, and LaMantia showed that the problem is polynomial time solvable for a subclass of
$(4P_1,C_4)$-free graphs \cite{FHHHL17}.
For a graph $G$ and a subset $S\subseteq V(G)$, we denote by $G[S]$ the subgraph of $G$ induced by $S$.
A graph is {\em chordal} if it is $C_t$-free for each $t\ge 4$.
\begin{theorem}
There exists a polynomial-time $2$-approximation algorithm for coloring $(4P_1,C_4)$-free graphs.
\end{theorem}

\begin{proof}
Let $G$ be a $(4P_1,C_4)$-free graph. Then $\overline{G}$ is $(2P_2,K_4)$-free.
By \autoref{thm:main}, we have that $\overline{G}$ can be partitioned into 4 stable sets.
So, $G$ can be partitioned into 4 cliques $K_i$ for $1\le i\le 4$, and this partition can be found in polynomial time.
Since $G$ is $C_4$-free, both $G[K_1\cup K_2]$ and $G[K_3\cup K_4]$ are chordal. 
It is well-known that the chromatic number of a chordal graph can be determined in linear time, see \cite{Go04} for example.
Therefore, the value $ \chi(G[K_1\cup K_2])+\chi(G[K_3\cup K_4])$ provides a 2-approximation
for $\chi(G)$.
\end{proof}

We now turn to the proof of \autoref{thm:main}.
The \emph{neighborhood} of a vertex $v$ in a graph $G$, denoted by $N_G(v)$, is the set of neighbors of $v$.
We simply write $N(v)$ if the graph $G$ is clear from the context.
Two nonadjacent vertices $u$ and $v$ in $G$ are \emph{comparable} if either $N(v)\subseteq N(u)$ or  $N(u)\subseteq N(v)$. 
Observe that if $N(u)\subseteq N(v)$, then $\chi(G-u)=\chi(G)$. 
Therefore, it suffices to prove \autoref{thm:main} for every connected $(2P_2,K_4)$-free graph with no pair of comparable vertices.
We do so by proving a number of lemmas below.
The idea is that we assume the occurrence of some induced subgraph $H$ in $G$
and then argue that the theorem holds in this case. Afterwards, we can assume that
$G$ is $H$-free in addition to being $(2P_2,K_4)$-free. We then pick a different induced subgraph as
$H$ and repeat. In the end, we are able to show that the theorem holds if $G$ contains a $C_5$
(see \autoref{lem:F1}-\autoref{lem:C5} below). Therefore, the remaining case is that $G$ is (odd hole, $K_4$)-free. In this case,
the theorem follows from a known result by Chudnovsky, Robertson, Seymour and Thomas \cite{CRST10} that every (odd hole, $K_4$)-free graph is $4$-colorable. 
This proves \autoref{thm:main}.

\begin{figure}[tb]
\centering
\begin{subfigure}{.5\textwidth}
\centering
\begin{tikzpicture}[scale=0.7]
\tikzstyle{vertex}=[draw, circle, fill=white!100, minimum width=4pt,inner sep=2pt]

\node [vertex] (v1) at (0,3) {$w$};
\node [vertex] (v2) at (3,1) {$2$};
\node [vertex] (v3) at (1.5,-1) {$6$};
\node [vertex] (v4) at (-1.5,-1) {$3$};
\node [vertex] (v5) at (-3,1) {$1$};
\draw 
(v1)--(v2)--(v3)--(v4)--(v5)--(v1);

\node[vertex] (y) at (1.5,0.25) {$4$};
\draw (v1)--(y)  (v2)--(y) (v3)--(y) (v5)--(y)  ;

\node[vertex] (z) at (-1.5,0.25) {$5$};
\draw (v4)--(z) (v5)--(z)  (v1)--(z) (v2)--(z);

\node at (0,-2) {$H_1$};
\end{tikzpicture}
\end{subfigure}%
\begin{subfigure}{.5\textwidth}
\centering
\begin{tikzpicture}[scale=0.7]
\tikzstyle{vertex}=[draw, circle, fill=white!100, minimum width=4pt,inner sep=2pt]

\node [vertex] (v1) at (0,3) {$1$};
\node [vertex] (v2) at (3,1) {$2$};
\node [vertex] (v3) at (1.5,-1) {$3$};
\node [vertex] (v4) at (-1.5,-1) {$4$};
\node [vertex] (v5) at (-3,1) {$5$};
\draw 
(v1)--(v2)--(v3)--(v4)--(v5)--(v1);

\node [vertex] (t) at (0,1) {$t$};
\draw (t)--(v5)  (t)--(v2)  (t)--(v3)  (t)--(v4);

\node at (0,-2) {$H_2$};
\end{tikzpicture}
\end{subfigure}
\caption{Two special graphs $H_1$ and $H_2$.}
\label{fig:F1F2}
\end{figure}
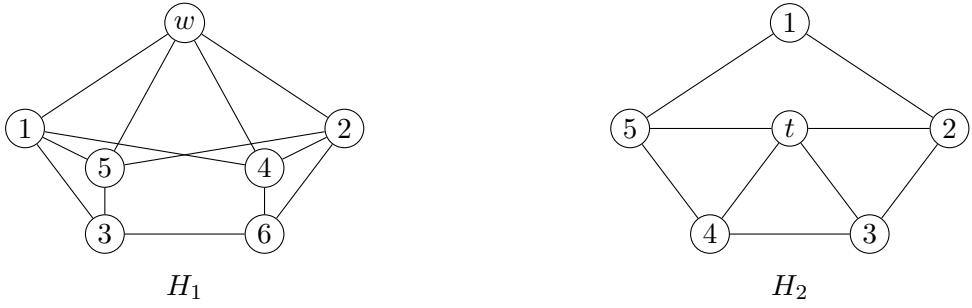

Th proof idea is based on a paper by Esperet et al.~\cite{ELMM13} who proved that every $(P_5,K_4)$-free graph is 5-colorable.
In particular, the graph $H_1$ (see \autoref{fig:F1F2}) that plays an important role in our proof was also used in \cite{ELMM13}.
However, to prove 4-colorability we need to use the argument of comparable vertices and 
extensively extend the structural analysis in \cite{ELMM13}.
The remainder of the paper is organized as follows. 
In \autoref{sec:pre} we present some preliminary results.
In \autoref{sec:F1} and \autoref{sec:F2} we prove \autoref{lem:F1} and  \autoref{lem:F2}, respectively.
We then prove \autoref{lem:W5} and \autoref{lem:C5} in \autoref{sec:W5C5}.
\section{Preliminaries}\label{sec:pre}

We present the structure around a five-cycle in $(2P_2,K_4)$-free graphs that will be used
in \autoref{sec:F2} and \autoref{sec:W5C5}.
Let $G$ be a $(2P_2,K_4)$-free graph and $C=12345$ be an induced $C_5$ of $G$.
All indices below are modulo 5. We partition  $V\setminus C$ into the following subsets:
\begin{equation*} \label{eq1}
\begin{split}
Z & = \{v\in V\setminus C: N_{C}(v)=\emptyset\}, \\
R_i  & = \{v\in V\setminus C: N_{C}(v)=\{i-1,i+1\}\}, \\
Y_i  & = \{v\in V\setminus C: N_{C}(v)=\{i-2, i, i+2\}\}, \\
F_i  & = \{v\in V\setminus C: N_{C}(v)=C\setminus \{i\}\}, \\
U & = \{v\in V\setminus C: N_{C}(v)=C\}. \\
\end{split}
\end{equation*}

\begin{lemma}\label{lem:partition}
Let $G$ be a $(2P_2,K_4)$-free graph and $C=12345$ be an induced $C_5$ of $G$.
Then $V(G)=C\cup Z\cup (\bigcup_{i=1}^{5}R_i)\cup (\bigcup_{i=1}^{5}Y_i)\cup (\bigcup_{i=1}^{5}F_i)\cup U$.
\end{lemma}

\begin{proof}
Suppose that there is a vertex $v\in V(G)\setminus C$ that does not belong to any of $Z$, $R_i$, $Y_i$, $F_i$ and $U$.
Note that $v$ has at least one and at most three neighbors on $C$. Moreover, these neighbors must be consecutive on $C$.
Without loss of generality, we may assume that $v$ is adjacent to $1$ and not adjacent to $3$ and $4$.
Now $34$ and $1v$ induce a $2P_2$.
\end{proof}

We now prove some structural properties of these sets.

\begin{enumerate}
\item {$Z\cup R_i$ is an independent set.}\label{itm:ZR}

If $Z\cup R_i$ contains an edge $xy$, then $xy$ and $(i-2)(i+2)$ induce a $2P_2$, a contradiction. \e

\item {$U\cup Y_i$ and $U\cup F_i$ are independent sets.}\label{itm:UYF}

If either $U\cup Y_i$ or $U\cup F_i$ contains an edge $xy$, then $\{x,y,i-2,i+2\}$ induces a $K_4$. \e

\item {$R_i$ and $R_{i+1}$ are complete.}\label{itm:RiRi+1}

It suffices to prove for $i=1$. If $r_1\in R_1$ and $r_2\in R_2$ are not adjacent, then $5r_1$ and $3r_2$
induce a $2P_2$. \e

\item {$Y_i$ and $Y_{i+1}$ are complete.}\label{itm:YiYi+1}

It suffices to prove for $i=1$. If $y_1\in Y_1$ and $y_2\in Y_2$ are not adjacent, then $5y_2$ and $3y_1$
induce a $2P_2$. \e

\item {$R_i$ and $Y_i$ are complete.}\label{itm:RiYi}

 It suffices to prove for $i=1$. If $r_1\in R_1$ and $y_1\in Y_1$ are not adjacent, then $5r_1$ and $3y_1$
induce a $2P_2$. \e

\item {Either $R_i$ and $Y_{i+1}$ are anti-complete or $R_{i+1}$ and $Y_i$ are anti-complete.}\label{itm:RiYi+1Ri+1Yi}

 Suppose, by contradiction, that there exist vertices $r_i\in R_i$, $r_{i+1}\in R_{i+1}$,
$y_i\in Y_i$, $y_{i+1}\in Y_{i+1}$ such that $r_i$ and $r_{i+1}$ are adjacent to $y_{i+1}$ and $y_i$, respectively.
Then it follows from \ref{itm:RiRi+1}, \ref{itm:YiYi+1} and \ref{itm:RiYi} that $\{r_i,r_{i+1},y_i,y_{i+1}\}$
induces a $K_4$. \e

\item {Each vertex in $Y_i$ is anti-complete to either $Y_{i-2}$ or $Y_{i+2}$.}\label{itm:YiYi-2Yi+2}

It suffices to prove for $i=1$. If $y_1\in Y_1$ is adjacent to a vertex $y_i\in Y_i$
for $i=3,4$, then $\{1,y_1,y_3,y_4\}$ induces a $K_4$ by \ref{itm:YiYi+1}. \e

\item {$F_i$ is complete to $Y_{i-2}\cup Y_{i+2}$ and anti-complete to $Y_{i-1}\cup Y_i\cup Y_{i+1}$.}\label{itm:FY}

It suffices to prove for $i=5$. Let $f\in F_5$. 
Recall that $f$ is adjacent to $1,2,3,4$ but not adjacent to $5$ by the definition of $F_5$.
Suppose first that $f$ is not adjacent to a vertex $y\in Y_2\cup Y_3$.
Note that $y$ is adjacent to $5$ by the definition of $Y_2$ and $Y_3$.
Now either $3f$ or $2f$ forms a $2P_2$ with $5y$ depending on whether $y\in Y_2$ or $y\in Y_3$.
This proves the first part of \ref{itm:FY}. Suppose now that $f$ is adjacent to a vertex $y\in Y_i$
for some $i\in \{1,4,5\}$. Since $i\notin \{2,3\}$, it follows that $5\notin \{i-2,i+2\}$.
Therefore, $f$ is adjacent to $i-2$ and $i+2$. This implies that $\{f,y,i-2,i+2\}$ induces a $K_4$. 
This proves the second part of \ref{itm:FY}. \e

\item {$F_i$ is complete to $R_{i-1}\cup R_{i+1}$.}\label{itm:FR}

It suffices to prove $i=5$. If $f\in F_5$ is not adjacent to $r\in R_1\cup R_4$, then
either $f3$ or $f2$ forms a $2P_2$ with $5r$ depending on whether $r\in R_1$ or $r\in R_4$. \e

\item {If $U\neq \emptyset$, then $Y_i$ and $Y_{i+2}$ are anti-complete.}\label{itm:YiYi+2}

 Let $u\in U$. If $y_i\in Y_i$ and $y_{i+2}\in Y_{i+2}$ are adjacent,
then $y_iy_{i+2}$ and $u(i+1)$ induce a $2P_2$ since $u$ is adjacent to neither $y_i$ nor $y_{i+2}$
by \ref{itm:UYF}, a contradiction. \e

\item {Either $F_i$ or $F_{i+2}$ is empty.}\label{itm:FiFi+2}

 It suffices to prove for $i=3$. Suppose that $F_i$ contains a vertex $f_i\in F_i$
for $i=3,5$. Then either $3f_5$ and $5f_3$ induce a $2P_2$ or $\{1,2,f_3,f_5\}$ induces a $K_4$
depending on whether $f_3$ and $f_5$ are nonadjacent or not. \e 

\item {If $G$ is $H_1$-free, then the following holds:
if $F_i\neq \emptyset$, then $R_{i+1}$ is anti-complete to $Y_{i+2}\cup Y_i$
and $R_{i-1}$ is anti-complete to $Y_{i-2}\cup Y_i$.}\label{itm:Ri+1YiYi+2}

It suffices to prove for $i=5$. Let $f\in F_5$. 
Suppose, by contradiction, that there exists vertices $r\in R_1$ and $y\in Y_2\cup Y_5$ such that
$r$ and $y$ are adjacent. Note that $f$ is adjacent to $r$ by \ref{itm:FR}.
If $y\in Y_2$, then $f$ is adjacent to $y$ by \ref{itm:FY} and this implies that
$\{f,y,r,2\}$ induces a $K_4$. If $y\in Y_5$, then $f$ is not adjacent to $y$ by \ref{itm:FY}
and this implies that $C\cup \setminus \{1\}\cup \{f,y,r\}$ induces an $H_1$ (see \autoref{fig:F1F2}). 
This proves that $R_1$ is anti-complete to $Y_2\cup Y_5$. The proof for the second part is symmetric. \e

\item {Each vertex in $R_i$ is anti-complete to either $Y_{i+1}$ or $Y_{i+2}$.
By symmetry, each vertex in $R_i$ is anti-complete to either $Y_{i-1}$ or $Y_{i-2}$}\label{itm:RiYi+1Yi+2}

Suppose, by contradiction, that there exists a vertex $r_i\in R_i$ such that
$r_i$ is adjacent to a vertex $y_{i+1}\in Y_{i+1}$ and a vertex  $y_{i+2}\in Y_{i+2}$.
By \ref{itm:YiYi+1}, $y_{i+1}$ and $y_{i+2}$ are adjacent. This implies that
$\{r_i,y_{i+1},y_{i+2},i-1\}$ induces a $K_4$. \e
\end{enumerate}

\section{Eliminate $H_1$}\label{sec:F1}

In this section we show that our main theorem, \autoref{thm:main}, holds when $G$ is connected, has no pair of comparable vertices, and contains $H_1$ as an induced subgraph.

\begin{lemma}\label{lem:F1}
Let $G$ be a connected $(2P_2,K_4)$-free graph with no pair of comparable vertices.
If $G$ contains an induced $H_1$, then $\chi(G)\le 4$.
\end{lemma}

\begin{proof}
Let $H=C\cup \{w\}$ be an induced $H_1$ in $G$ where $C=\{1,2,3,4,5,6\}$ induces a
$\overline{C_6}$ such that $ij$ is an edge if and only if $|i-j|\neq 1$, and $w$ is adjacent to 1, 2, 4 and 5
(See \autoref{fig:F1F2}).
All the indices below are modulo 6. We partition $V(G)$ into following subsets:
\begin{equation*} \label{eq1}
\begin{split}
Z & = \{v\in V\setminus C: N_{C}(v)=\emptyset\}, \\
D_{i,i+1}  & = \{v\in V\setminus C: N_{C}(v)=\{i,i+1\}\}, \\
T_i  & = \{v\in V\setminus C: N_{C}(v)=\{i-1,i,i+1\}\}, \\
F_{i,i+1}  & = \{v\in V\setminus C: N_{C}(v)=\{i-1, i, i+1, i+2\}\}, \\
W  & = \{v\in V\setminus C: N_{C}(v)=N_C(w)=\{1,2,4,5\}\}. \\
\end{split}
\end{equation*}
Let $D=\bigcup_{i=1}^{6}D_{i,i+1}$,  $T=\bigcup_{i=1}^{6}T_{i}$ and $F=\bigcup_{i=1}^{6}F_{i,i+1}$.
Without loss of generality, we assume $H$ has been chosen such that $|T|+|F|$ is maximized.
We first show that $V(G)=C\cup Z\cup D\cup T\cup F\cup W$.

\begin{enumerate}
\emitem {There is no vertex $v\in V\setminus C$ such that $v$ is adjacent to $i$ but adjacent to neither $i-1$ nor $i+1$
for any $1\le i\le 6$.}\label{itm:1}

Suppose that such a vertex $v$ exists. Then it follows that $vi$ and $(i-1)(i+1)$ induce a $2P_2$. \e

\emitem{If a vertex in $V\setminus C$ has at most two neighbors on $C$, then $v\in Z\cup D$.}\label{itm:2} 

Suppose not. Let $v\in V\setminus C$ that has at most two neighbors on $C$ and $v\notin Z\cup D$.
Then either $v$ has exactly one neighbor on $C$ or has two neighbors on $C$ that are not consecutive.
By symmetry, we may assume that $v$ is adjacent to $1$ but not adjacent to $2$ and $6$. This contradicts
\ref{itm:1}. \e

\emitem{If a vertex $v\in V\setminus C$ that has exactly three neighbors on $C$, then $v\in T$.}\label{itm:3}

Suppose not.  Let $v\in V\setminus C$ that has exactly at three neighbors on $C$.
By symmetry, we may assume that $v$ is adjacent to $1$.
It follows from \ref{itm:1} that $v$ is adjacent to either $2$ or $6$, say 2. 
If $v$ is not adjacent to $3$ or $6$, then it contradicts
\ref{itm:1} for $i=4$ or $i=5$. Therefore, $v\in T_1$ or $v\in T_2$. \e

\emitem{If a vertex $v\in V\setminus C$ that has exactly four neighbors on $C$, then $v\in F\cup W$.}\label{itm:4}

By \ref{itm:1}, $v$ must have two consecutive neighbors on $C$.
If $v$ has three consecutive neighbors on $C$, then all four neighbors must be consecutive
by \ref{itm:1} and so $v\in F$. Now $N_C(v)=\{i,i+1,i+3,i+4\}$ for some $i$. If $i=1$, then $v\in W$.
Suppose that $i=2$ (and the case $i=3$ is symmetric). Then either $w1$ and $v6$ induce
a $2P_2$ or $\{w,v,2,5\}$ induces a $K_4$, depending on whether $w$ and $v$ are nonadjacent or not. \e

\emitem {There is no vertex in $V\setminus C$ that has more than four neighbors.}\label{itm:5}

Suppose not. Let $v\in V\setminus C$ have at least five neighbors on $C$. 
By symmetry, we may assume that $v$ is adjacent to $i$ for each $1\le i\le 5$.
Then $\{1,3,5,v\}$ induces a $K_4$. \e
\end{enumerate}

It follows from \ref{itm:2}-\ref{itm:5} that $V(G)=C\cup Z\cup D\cup T\cup F\cup W$.
Note that each of the subsets defined is an independent set since $G$ is $(2P_2,K_4)$-free.
We further investigate the adjacency among those subsets.
\begin{enumerate}
\setcounter{enumi}{5}

\emitem {The set $W$ is anti-complete to $Z$.}\label{itm:WZ}

If $w\in W$ and $z\in Z$ are adjacent, then $wz$ and $36$ induce a $2P_2$, a contradiction. \e

\emitem {The set $W$ is complete to $D_{i,i+1}$ for $i\in \{2,3,5,6\}$ and anti-complete to $D_{i,i+1}$ for $i\in \{1,4\}$.}\label{itm:WD}

Suppose that $w\in W$ is not adjacent some vertex $d\in D_{i,i+1}$ for some $i\in \{2,3,5,6\}$.
By symmetry, we may assume that $i=2$. Then $d3$ and $w4$ induce a $2P_2$, a contradiction.  
Suppose that $w\in W$ is adjacent some vertex $d\in D_{1,2}\cup D_{4,5}$.
Then $dw$ and $36$ induce a $2P_2$, a contradiction.  \e

\emitem {The set $W$ is complete to $T_1\cup T_2\cup T_4\cup T_5$ and anti-complete to $T_3\cup T_6$.}\label{itm:WT}

Suppose that $w\in W$ is not adjacent some vertex $t\in T_{i}$ for some $i\in \{1,2,4,5\}$.
By symmetry, we may assume that $i=1$. Then $t6$ and $w5$ induce a $2P_2$.
Suppose that $w\in W$ is adjacent some vertex $t\in T_{i}$ for some $i\in \{3,6\}$.
By symmetry, we may assume that $i=3$. Then $\{w,t,2,4\}$ induces a $K_4$. \e

\emitem {The set $W$ is anti-complete to $F_{i,i+1}$ for $i\in \{2,3,5,6\}$ and complete to $F_{i,i+1}$ for $i\in \{1,4\}$.}\label{itm:WF}

Suppose that $w\in W$ is adjacent some vertex $f\in F_{i,i+1}$ for some $i\in \{2,3,5,6\}$.
By symmetry, we may assume that $i=2$. Then $\{f,w,1,4\}$ induces a $K_4$. 
Suppose that $w\in W$ is not adjacent some vertex $f\in F_{i,i+1}$ for some $i\in \{1,4\}$.
By symmetry, we may assume that $i=1$.  Then $6f$ and $5w$ induce a $2P_2$. \e

\emitem {The set $Z$ is anti-complete to $D\cup T\cup (F\setminus (F_{1,2}\cup F_{4,5}))$.}\label{itm:Z}

Suppose that $z\in Z$ is adjacent to some vertex $x\in D\cup T\cup (F\setminus (F_{1,2}\cup F_{4,5}))$.
If $x\in D\cup T$, then there exists a vertex $i\in C$ such that $x$ is not adjacent to $i-1$ and $i+1$.
Then $zx$ and $(i-1)(i+1)$ induce a $2P_2$.  If $x\in F_{i,i+1}$ for some $i=2,3,5,6$,
then $xw\notin E$ by \ref{itm:WF}. Moreover, there exists a vertex $j\in N_C(w)$
such that $xj\notin E$. Then $wj$ and $zx$ induce a $2P_2$. \e

\end{enumerate}

It follows from and \ref{itm:WZ} and \ref{itm:Z} that any vertex in $Z$ has neighbors only in $F_{1,2}\cup F_{4,5}$.
On the other hand, $w$ is complete to $F_{1,2}\cup F_{4,5}$ by \ref{itm:WF}.
Since $G$ contains no pair of comparable vertices,  it follows that $Z=\emptyset$.

\begin{enumerate}
\setcounter{enumi}{10}

\emitem{For each $i$, $D_{i,i+1}$ is anti-complete to $D_{i+1,i+2}$, complete to $D_{i+2,i+3}$ and anti-complete to $D_{i+3,i+4}$.}
\label{itm:DiDj}

By symmetry, it suffices to prove the claim for $i=1$. 
Let $d\in D_{1,2}$. If $d$ is adjacent to $d'\in D_{2,3}$, then 46 and $dd'$ induce a $2P_2$.
 If $d$ is not adjacent to $d'\in D_{3,4}$, then $2d$ and $3d'$ induce a $2P_2$.  
 If $d$ is adjacent to $d'\in D_{4,5}$, then 36 and $dd'$ induce a $2P_2$. \e

\emitem {For each $i$, $F_{i,i+1}$ is anti-complete to $F_{i+1,i+2}\cup F_{i+3,i+4}$ and complete to $F_{i+2,i+3}$.}
\label{itm:FiFj}

By symmetry, it suffices to prove the claim for $i=1$. Let $f\in F_{1,2}$.
If $f$ is adjacent to a vertex $f'\in F_{2,3}$, then $\{1,3,f,f'\}$ induces a $K_4$.
If $f$ is not adjacent to a vertex $f'\in F_{3,4}$, then $5f'$ and $6f$ induce a $2P_2$.
If $f$ is adjacent to a vertex $f'\in F_{4,5}$, then $\{3,6,f,f'\}$ induces a $K_4$. \e

\emitem{The sets $T_i$ and $T_{i+1}$ are anti-complete for $i\in \{1,4\}$.}\label{itm:TiTi+1}

By symmetry, it suffices to prove this for $i=1$.
If $t_1\in T_1$ and $t_2\in T_2$ are adjacent, then $w$ is adjacent to both $t_1$ and $t_2$ by \ref{itm:WT}.
But now $\{t_1,t_2,w,1\}$ induces a $K_4$. \e

\emitem{The sets $T_3$ and $T_1\cup T_5$ are complete. By symmetry, $T_6$ and $T_2\cup T_4$ are complete.}\label{itm:TiTi+2}

Suppose that $t_3\in T_3$ is not adjacent to some vertex $t\in T_1\cup T_5$. 
By \ref{itm:WT}, $w$ is adjacent to $t$ but not to $t_3$. Then $3t_3$ and $wt$ induce a $2P_2$, a contradiction. \e

\emitem {The sets $T_i$ and $T_{i+3}$ are complete for each $1\le i\le 6$.}\label{itm:TiTi+3}

By symmetry, it suffices to prove this for $i=1$.
If $t_1\in T_1$ and $t_4\in T_4$ are not adjacent, then $2t_1$ and $3t_4$ induce a $2P_2$.\e

\emitem{For each $i$, $D_{i,i+1}$ is anti-complete to $T_{i-1}\cup T_i\cup T_{i+1}\cup T_{i+2}$ and 
complete to $T_{i+3}\cup T_{i+4}$.} \label{itm:DT}

We note that $D_{1,2}$ and $D_{4,5}$ are symmetric, and $D_{2,3}$, $D_{3,4}$, $D_{5,6}$ and $D_{6,1}$ are symmetric.
So, it suffices to prove the claim for $D_{1,2}$ and $D_{2,3}$.

Let $d\in D_{1,2}$. Suppose that $d$ is adjacent to some vertex $t\in T_6\cup T_1\cup T_2\cup T_3$.
By symmetry, we may assume that $i\in \{1,3\}$. If $i=1$, then $td$ and $35$ induce a $2P_2$.
If $i=3$, then $w$ is not adjacent to $d$ and $t$ by \ref{itm:WD} and \ref{itm:WT}.
Then $dt$ and $w5$ induce a $2P_2$. Now suppose that $d$ is not adjacent to some vertex $t\in T_4\cup T_5$.
By symmetry, we may assume that $t\in T_4$. Then $d2$ and $t3$ induce a $2P_2$.  This proves the claim
for $D_{1,2}$.

Let $d\in D_{2,3}$.  Suppose that $d$ is adjacent to some vertex $t\in  T_2\cup T_3$.
By symmetry, we may assume that $t\in T_2$. Then $dt$ and $46$ induce a $2P_2$.
Suppose that $d$ is not adjacent to some vertex $t\in  T_5\cup T_6$.
By symmetry, we may assume that $t\in T_5$. Then $d3$ and $t4$ induce a $2P_2$.

By \ref{itm:WD} and \ref{itm:WT}, $\{2,w\}$ is complete to $D_{2,3}\cup T_1$.
It follows from $K_4$-freeness of $G$ that $D_{2,3}$ is anti-complete to $T_1$.
It remains to show that $D_{2,3}$ is anti-complete to $T_4$.
Suppose that $d$ is adjacent to some vertex $t_4\in T_4$.
Note that $C'=C\setminus \{1\}\cup \{t_4\}$ induces a $\overline{C_6}$ and 
$H'=C'\cup \{w\}$ induces a subgraph isomorphic to $H_1$.
By \ref{itm:TiTi+1} and \ref{itm:TiTi+2}, all vertices in $T_1\cup T_4\cup T_5\cup T_6$ remain to be $T$-vertices with respect to $C'$.
Moreover, all vertices in $T_3\cup F$ remain to be $F$-vertices or $T$-vertices.
By the choice of $C$, there exists a vertex $t\in T_2$ that is not adjacent to $t_4$.
Then $dt_4$ and $1t_2$ induce a $2P_2$, a contradiction. This proves the claim for $D_{2,3}$.\e

\emitem {For each $i$, $F_{i,i+1}$ is anti-complete to $T_i\cup T_{i+1}$ and complete to $T_{i+3}\cup T_{i+4}$}
\label{itm:TF}

By symmetry of $C$, it suffices to prove this for $i=1$.
Let $f\in F_{1,2}$. If $f$ is adjacent to some vertex $t\in T_1\cup T_2$, then either $\{6,2,f,t\}$
or $\{1,3,f,t\}$ induces a $K_4$ depending on whether $t\in T_1$ or $t\in T_2$.
Suppose that $f$ is not adjacent to some vertex $t\in T_4\cup T_5$.
By symmetry, we may assume that $t\in T_4$. Then $6f$ and $5t$ induce a $2P_2$, a contradiction. \e

\emitem {The sets $F_{i,i+1}$ and $T_{i-1}$ are complete for $i\in \{2,5\}$,
and $F_{i,i+1}$ and $T_{i+2}$ are complete for $i\in \{3,6\}$. }
\label{itm:TFextra}

Let $f\in F_{i,i+1}$ and $t\in T_i$ be nonadjacent. 
By \ref{itm:WF} and \ref{itm:WT}, $w$ is adjacent to $t$ but not $f$.
It can be readily checked that in each of the cases $wt$ and $f3$ or  $wt$ and $f6$
induce a $2P_2$. \e 

\emitem {
The set $D_{1,2}$ is anti-complete to $F_{6,1}\cup F_{2,3}$ and complete to $F_{45}$.

The set $D_{4,5}$ is anti-complete to $F_{3,4}\cup F_{5,6}$ and complete to $F_{12}$.

The set $D_{2,3}$ is anti-complete to $F_{1,2}$ and complete to $F_{5,6}\cup F_{6,1}$.
 
The set $D_{3,4}$ is anti-complete to $F_{4,5}$ and complete to $F_{5,6}\cup F_{6,1}$.

The set $D_{6,1}$ is anti-complete to $F_{1,2}$ and complete to $F_{2,3}\cup F_{3,4}$.

The set $D_{5,6}$ is anti-complete to $F_{4,5}$ and complete to $F_{2,3}\cup F_{3,4}$.
}\label{itm:DF}

Note that $D_{1,2}$ and $D_{4,5}$ are symmetric, and $D_{2,3}$, $D_{3,4}$, $D_{5,6}$ and $D_{6,1}$ are symmetric.
So, it suffices to prove the claim for $D_{1,2}$ and $D_{2,3}$.
Let $d\in D_{1,2}$. If $d$ is adjacent to some vertex $f\in F_{6,1}\cup F_{2,3}$,
then $w$ is not adjacent to $d$ and $f$ by \ref{itm:WD} and \ref{itm:WF}.
Now $df$ and $w4$ or $df$and $w5$ induce a $2P_2$ depending on whether $f\in F_{6,1}$ or $f\in F_{2,3}$.
If $d$ is not adjacent to some vertex $f\in F_{4,5}$, then $d2$ and $f3$ induce a $2P_2$.
This proves the claim for $D_{1,2}$.

Now let $d\in D_{2,3}$. By \ref{itm:WD}, it follows that $wd\in E$.
If $d$ is adjacent to a vertex $f\in F_{1,2}$, then $\{d,f,2,w\}$ induces a $K_4$ by \ref{itm:WF}.
If $d$ is not adjacent to a vertex $f\in F_{5,6}\cup F_{6,1}$, then $6f$ and $wd$ induce a $2P_2$ by \ref{itm:WF}.
This proves the claim for $D_{2,3}$. \e

\end{enumerate}

We proceed with a few claims that help to show that certain sets are empty.
\begin{claim}\label{cla:D12D45}
Either $D_{1,2}$ or $D_{4,5}$ is empty.
\end{claim}

\begin{proof}[{\bf Proof of \autoref{cla:D12D45}}]
Suppose not. Let $d_{12}\in D_{1,2}$ and $d_{45}\in D_{4,5}$.
By \ref{itm:WD}-\ref{itm:DF}, $N(d_{12})\subseteq N(w)$ unless
$d_{12}$ has a neighbor $f\in F_{3,4}\cup F_{5,6}$.
Similarly, $N(d_{45})\subseteq N(w)$ unless
$d_{45}$ has a neighbor $f'\in F_{3,4}\cup F_{5,6}$.
By \ref{itm:DiDj} and \ref{itm:DF}, $d_{12}f$ and $d_{45}f'$ induce a $2P_2$, a contradiction.
\end{proof}

\begin{claim}\label{cla:T1T2T4T5}
Each vertex in $T_1$ has a non-neighbor in $T_5$ and each vertex in $T_5$ has a non-neighbor in $T_1$.
By symmetry, each vertex in $T_2$ has a non-neighbor in $T_4$ and each vertex in $T_4$ has a non-neighbor in $T_2$.
\end{claim}

\begin{proof}[{\bf Proof of \autoref{cla:T1T2T4T5}}]
Let $t_1\in T_1$. Let 
\[X=\{6,1,2\}\cup W\cup D_{3,4}\cup D_{4,5}\cup T_3\cup T_4\cup F_{2,3}\cup F_{3,4}\cup F_{4,5}.\]
Note that $N(4)=X\cup T_5\cup F_{5,6}$ and $N(t_1)\subseteq X\cup T_5\cup F_{5,6}\cup T_6$ by the properties
we have proved. Since $G$ contains no pair of comparable vertices, $t_1$ has a neighbor $t_6\in T_6$
and there exists a vertex $t\in N(4)\setminus N(t_1)$. Clearly, $t\in F_{5,6}\cup T_5$.
If $t\in F_{5,6}$, then $4t$ and $t_1t_6$ induce a $2P_2$ since $F_{56}$ and $T_6$ are anti-complete by \ref{itm:TF}.
This shows that $t_1$ has a non-neighbor $t\in T_5$. By symmetry, each vertex in $T_5$ has a non-neighbor in $T_1$.
\end{proof}

\begin{claim}\label{cla:T3T6}
Each vertex in $T_6$ has a neighbor in $T_1\cup T_5$.
By symmetry, each vertex in $T_3$ has a neighbor in $T_2\cup T_4$.
\end{claim}

\begin{proof}[{\bf Proof of \autoref{cla:T3T6}}]
Let $t_6\in T_6$. Let
\[X=\{5,6,1\}\cup D_{2,3}\cup D_{3,4}\cup T_2\cup T_3\cup T_4\cup F_{2,3}\cup F_{3,4}.\]
Note that $N(3)=X\cup F_{1,2}\cup F_{4,5}$ and $N(t_6)\subseteq X\cup T_1\cup T_5\cup F_{12}\cup F_{45}$.
Since $G$ contains no pair of comparable vertices, $t_6$ has a neighbor in $T_1\cup T_5$.
\end{proof}

\begin{claim}\label{cla:D5661T2T4complete}
If $D_{5,6}\cup D_{6,1}\neq \emptyset$, then $T_2$ and $T_4$ are complete.
By symmetry, if $D_{2,3}\cup D_{3,4}\neq \emptyset$, then $T_1$ and $T_5$ are complete.
\end{claim}

\begin{proof}[{\bf Proof of \autoref{cla:D5661T2T4complete}}]
Let $d\in D_{5,6}\cup D_{6,1}$. Suppose that $t_2\in T_2$ and $t_4\in T_4$ are not adjacent.
If $d\in D_{5,6}$, then $dt_2\in E$ and $dt_4\notin E$ by \ref{itm:DT}. Thus, $dt_2$ and $4t_4$
induce a $2P_2$. If $d\in D_{6,1}$, then $dt_4\in E$ and $dt_2\notin E$ by \ref{itm:DT}. Thus, $dt_4$ and $2t_2$
induce a $2P_2$. 
\end{proof}

\begin{claim}\label{cla:F61F12F23}
One of $F_{6,1}$, $F_{1,2}$ and $F_{2,3}$ is empty.
By symmetry, one of $F_{3,4}$, $F_{4,5}$ and $F_{5,6}$ is empty.
\end{claim}

\begin{proof}[{\bf Proof of \autoref{cla:F61F12F23}}]
Suppose that $f_{61}\in F_{6,1}$, $f_{12}\in F_{1,2}$, and $f_{23}\in F_{2,3}$.
Then $f_{61}f_{23}$ and $f_{12}w$ induce a $2P_2$ by \ref{itm:WF} and \ref{itm:FiFj}.  
\end{proof}

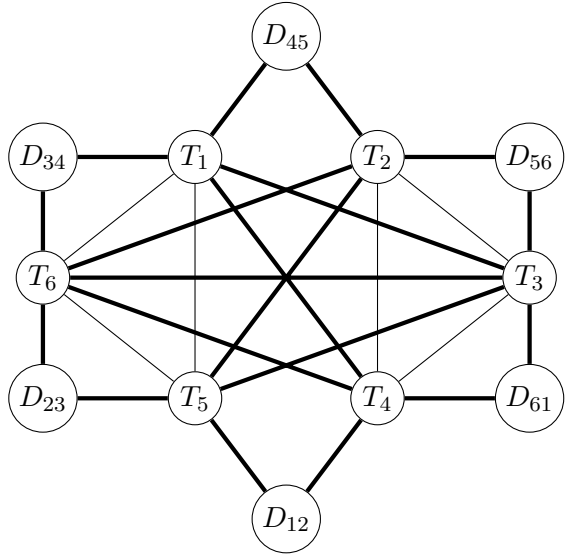
\begin{figure}[t]
\centering

\begin{tikzpicture}[scale=0.8]
\tikzstyle{vertex}=[draw, circle, fill=white!100, minimum width=4pt,inner sep=2pt]

\node[vertex] (v1) at (-1.5,2) {$T_1$};
\node[vertex] (v2) at (1.5,2) {$T_2$};
\node[vertex] (v3) at (4,0) {$T_3$};
\node[vertex] (v4) at (1.5,-2) {$T_4$};
\node[vertex] (v5) at (-1.5,-2) {$T_5$};
\node[vertex] (v6) at (-4,0) {$T_6$};

\draw (v2)--(v3)--(v4)--(v2);
\draw (v1)--(v6)--(v5)--(v1);
\draw[ultra thick] (v1)--(v4) (v2)--(v5) (v3)--(v6);
\draw[ultra thick] (v3)--(v1) (v3)--(v5) (v6)--(v2) (v6)--(v4);

\node[vertex] (d45) at (0,4) {$D_{45}$};
\draw[ultra thick] (d45)--(v1) (d45)--(v2);
\node[vertex] (d12) at (0,-4) {$D_{12}$};
\draw[ultra thick] (d12)--(v4) (d12)--(v5);
\node[vertex] (d23) at (-4,-2) {$D_{23}$};
\draw[ultra thick] (d23)--(v5) (d23)--(v6);
\node[vertex] (d34) at (-4,2) {$D_{34}$};
\draw[ultra thick] (d34)--(v6) (d34)--(v1);
\node[vertex] (d56) at (4,2) {$D_{56}$};
\draw[ultra thick] (d56)--(v2) (d56)--(v3);
\node[vertex] (d61) at (4,-2) {$D_{61}$};
\draw[ultra thick] (d61)--(v3) (d61)--(v4);

\end{tikzpicture}

\caption{The adjacency among $T_i$ and $D_{i,i+1}$.  A thick line between two sets means that the two sets are complete,
a thin line means the edges between the two sets can be arbitrary, and no line means that the two sets are anti-complete.
For clarity, edges between two $D_{i,i+1}$ are not shown.}\label{fig:DT}

\end{figure}

By \autoref{cla:D12D45}, we may assume that $D_{4,5}=\emptyset$.
It follows from \ref{itm:TiTi+1}, \ref{itm:TiTi+2} and \ref{itm:TiTi+3} that
either $T_1$ and $T_5$ are complete or $T_2$ and $T_4$ are complete for otherwise $G$ would contain a $2P_2$ (see \autoref{fig:DT}).
By symmetry, we may assume that $T_1$ and $T_5$ are complete.
It then follows from \autoref{cla:T1T2T4T5} and \autoref{cla:T3T6} that
$T_1\cup T_5\cup T_6=\emptyset$.

If $D_{5,6}\cup D_{6,1}\neq \emptyset$, then $T_2\cup T_3\cup T_4=\emptyset$ 
due to \autoref{cla:T1T2T4T5}-\autoref{cla:D5661T2T4complete}.
In the following we shall use this fact without explicitly mentioning it. 
We divide our proof into four cases depending on whether $F_{1,2}$ and $F_{4,5}$ are empty or not.
One can verify that each of the partitions of $V(G)$ into 4 subsets in the following is a 4-coloring of $G$
using the properties we have proved. For convenience, we draw \autoref{fig:DF} to visulize the adjacency among $D_{i,i+1}$ and $F_{i,i+1}$.
From \autoref{fig:DF} it can be seen that if $T_2\cup T_3\cup T_4=\emptyset$, then we can use the symmetry of $H$ under its
automorphism $f:V(H)\rightarrow V(H)$ with $f(1)=2$, $f(2)=1$, $f(3)=6$, $f(4)=5$, $f(5)=4$, $f(6)=3$ and $f(w)=w$.

\begin{figure}[t]
\centering

\begin{tikzpicture}[scale=0.8]
\tikzstyle{vertex}=[draw, circle, fill=white!100, minimum width=4pt,inner sep=2pt]

\node[vertex] (f12) at (0,4) {$F_{12}$};
\node[vertex] (f23) at (4,2) {$F_{23}$};
\node[vertex] (f34) at (4,-2) {$F_{34}$};
\node[vertex] (f45) at (0,-4) {$F_{45}$};
\node[vertex] (f56) at (-4,-2) {$F_{56}$};
\node[vertex] (f61) at (-4,2) {$F_{61}$};

\draw[ultra thick] (f12)--(f34)--(f56)--(f12);
\draw[ultra thick] (f23)--(f45)--(f61)--(f23);

\node[vertex] (d12) at  (0,0) {$D_{12}$};
\draw[ultra thick] (d12)--(f45);
\draw (f34)--(d12)--(f56) (d12)--(f12);

\node[vertex] (d23) at  (-6,-5) {$D_{23}$};
\draw[ultra thick] (f56)--(d23)--(f61);
\draw (d23)--(f23)  (d23)--(f34); 
\draw[anchor=south] (d23)--(f45);

\node[vertex] (d34) at  (-6,5) {$D_{34}$};
\draw[ultra thick] (f56)--(d34)--(f61);
\draw (d34)--(f23)  (d34)--(f34); 
\draw[anchor=south] (d34)--(f12);

\node[vertex] (d56) at  (6,5) {$D_{56}$};
\draw[ultra thick] (f23)--(d56)--(f34);
\draw (d56)--(f56)  (d56)--(f61); 
\draw[anchor=south] (d56)--(f12);

\node[vertex] (d61) at  (6,-5) {$D_{61}$};
\draw[ultra thick] (f23)--(d61)--(f34);
\draw (d61)--(f56)  (d61)--(f61); 
\draw[anchor=south] (d61)--(f45);
\end{tikzpicture}

\caption{The adjacency among $F_{i,i+1}$ and $D_{i,i+1}$.  A thick line between two sets means that the two sets are complete,
a thin line means the edges between the two sets can be arbitrary, and no line means that the two sets are anti-complete.
For clarity, edges between two $D_{i,i+1}$ are not shown.}\label{fig:DF}

\end{figure}
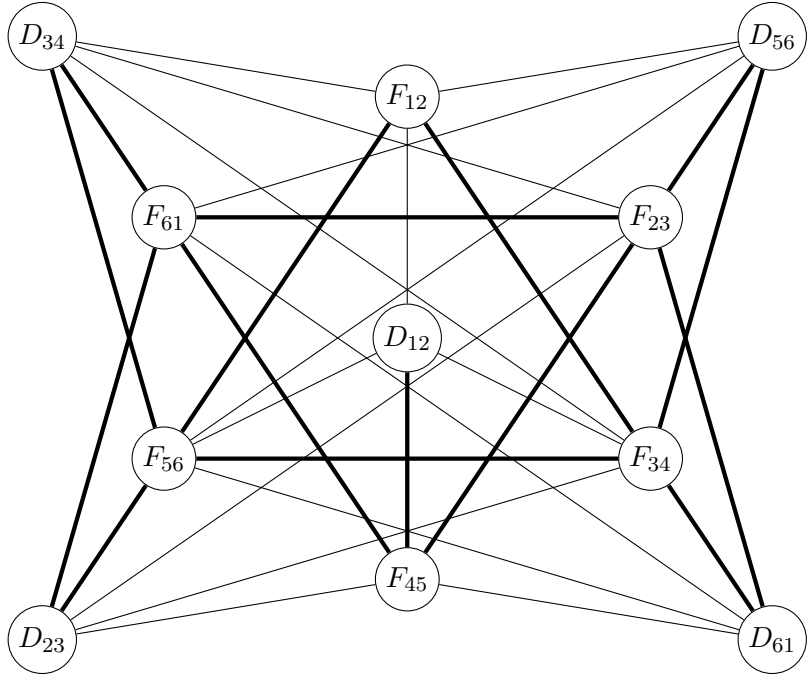

\medskip

\case{1} Both $F_{1,2}$ and $F_{4,5}$ are not empty. 
Let $f_{12}\in D_{1,2}$ and $f_{45}\in D_{4,5}$.
We first show that $F_{1,2}\cup F_{4,5}$ is anti-complete to $D_{2,3}\cup D_{3,4}\cup D_{5,6}\cup D_{6,1}$.
By symmetry, it suffices to show that $F_{1,2}\cup F_{4,5}$ is anti-complete to $D_{2,3}$.
Suppose that $d\in D_{2,3}$ and $f\in F_{1,2}\cup F_{4,5}$ are adjacent. By \ref{itm:DF}, $f\in F_{4,5}$.
Then $df$ and $1f_{12}$ induce a $2P_2$.
On the other hand,  it follows from \autoref{cla:F61F12F23} and \ref{itm:FiFj} that
at most one of $F_{2,3}$, $F_{3,4}$, $F_{5,6}$ and $F_{6,1}$ is not empty.

$\bullet$ If $F_{2,3}\neq \emptyset$, then $G$ has a 4-coloring:
\begin{align*}
& F_{4,5}\cup D_{2,3}\cup D_{3,4} \cup \{1\}\cup T_4,\\
& F_{2,3}\cup D_{1,2}\cup W\cup \{6\}\cup T_3,	\\
& F_{1,2}\cup \{4,5\}\cup T_2,		\\
& D_{5,6}\cup D_{6,1}\cup \{2,3\}. 	\\	
\end{align*}

$\bullet$  Suppose that $F_{6,1}\neq \emptyset$. 

If $D_{5,6}\cup D_{6,1}\neq \emptyset$, then $G$ has a 4-coloring:
\begin{align*}
& F_{4,5}\cup D_{5,6}\cup D_{6,1} \cup \{2\},\\
& F_{6,1}\cup D_{1,2}\cup W\cup \{3\},	\\
& F_{1,2}\cup \{4,5\},		\\
& D_{2,3}\cup D_{3,4}\cup \{1,6\}. 	\\	
\end{align*}

If $D_{5,6}\cup D_{6,1}= \emptyset$, then $G$ has a 4-coloring:
\begin{align*}
& F_{4,5}\cup \{1,2\}\cup T_4,\\
& F_{6,1}\cup D_{1,2}\cup W\cup \{3\},	\\
& F_{1,2}\cup \{4,5\}\cup T_2,		\\
& D_{2,3}\cup D_{3,4}\cup \{6\}\cup T_3. 	\\	
\end{align*}

$\bullet$  Suppose that $F_{3,4}\neq \emptyset$. Note first that no vertex $d\in D_{1,2}$ can have a neighbor in
both $F_{1,2}$ and $F_{3,4}$ for otherwise a neighbor of $d$ in $F_{1,2}$, a neighbor of $d$ in $F_{3,4}$,
$d$ and $2$ induce a $K_4$. Let $D'_{1,2}$ be the set of vertices in $D_{1,2}$ that are anti-complete to $F_{3,4}$
and $D''_{1,2}=D_{1,2}\setminus D'_{1,2}$. Then $G$ has a 4-coloring:
\begin{align*}
&F_{4,5}\cup D_{2,3}\cup D_{3,4} \cup \{1\}\cup T_4,\\
&F_{3,4}\cup D'_{1,2}\cup W\cup \{6\}\cup T_3,	\\
&F_{1,2}\cup D''_{1,2}\cup \{4,5\}\cup T_2,		\\
&D_{5,6}\cup D_{6,1}\cup \{2,3\}. 	\\	
\end{align*}

$\bullet$  Suppose that $F_{5,6}\neq \emptyset$.  Note first that no vertex $d\in D_{1,2}$ can have a neighbor in
both $F_{1,2}$ and $F_{5,6}$ for otherwise a neighbor of $d$ in $F_{1,2}$, a neighbor of $d$ in $F_{5,6}$,
$d$ and 1 induce a $K_4$. Let $D'_{1,2}$ be the set of vertices in $D_{1,2}$ that are anti-complete to $F_{5,6}$
and $D''_{1,2}=D_{1,2}\setminus D'_{1,2}$. By \ref{itm:TF} and \ref{itm:TFextra}, $F_{5,6}$ and $T_3\cup T_4$ are complete.
Since $G$ is $K_4$-free, $T_3$ and $T_4$ are anti-complete.
Then $G$ has a 4-coloring:
\begin{align*}
&F_{4,5}\cup D_{5,6}\cup D_{6,1} \cup \{2\},\\
&F_{5,6}\cup D'_{1,2}\cup W\cup \{3\},	\\
&F_{1,2}\cup D''_{1,2}\cup \{4,5\}\cup T_2,		\\
&D_{2,3}\cup D_{3,4}\cup \{1,6\}\cup T_3\cup T_4. 	\\	
\end{align*}

\case{2} Both $F_{1,2}$ and $F_{4,5}$ are empty.
By \ref{itm:FiFj} and the fact that $G$ is $2P_2$-free, one of $F_{2,3}$, $F_{3,4}$, $F_{5,6}$ and $F_{6,1}$ is empty.
 By \ref{itm:DiDj}, \ref{itm:DF}, \ref{itm:FiFj} and $K_4$-freeness of $G$,
either $D_{5,6}$ and $F_{5,6}$ are anti-complete or $D_{3,4}$ and $F_{3,4}$ are anti-complete.

$\bullet$ Suppose that $F_{6,1}=\emptyset$.

If $D_{5,6}$ and $F_{5,6}$ are anti-complete, then $G$ has a 4-coloring:
\begin{align*}
&F_{2,3}\cup F_{3,4}\cup W\cup \{6\}\cup T_3,\\
&F_{5,6}\cup D_{5,6}\cup \{2,3\},	\\
&D_{1,2}\cup D_{6,1}\cup \{4,5\}\cup T_2,		\\
&D_{2,3}\cup D_{3,4}\cup \{1\}\cup T_4. 	\\	
\end{align*}

Now assume that $D_{3,4}$ and $F_{3,4}$ are anti-complete.

If $D_{5,6}\cup D_{6,1}\neq \emptyset$, then $G$ has a 4-coloring:
\begin{align*}
&F_{2,3}\cup F_{5,6}\cup W,\\
&F_{3,4}\cup D_{3,4}\cup \{6,1\},	\\
&D_{1,2}\cup D_{2,3}\cup \{4,5\},		\\
&D_{5,6}\cup D_{6,1}\cup \{2,3\}. 	\\	
\end{align*}

If $D_{5,6}\cup D_{6,1}= \emptyset$, then $G$ has a 4-coloring:
\begin{align*}
&F_{2,3}\cup D_{1,2}\cup W\cup \{6\}\cup T_3,\\
&F_{3,4}\cup D_{3,4}\cup \{1\}\cup T_4,	\\
&F_{5,6}\cup \{2,3\},		\\
&D_{2,3}\cup \{4,5\}\cup T_2. 	\\	
\end{align*}

$\bullet$ Suppose that $F_{2,3}=\emptyset$. 

Suppose first that $D_{3,4}$ and $F_{3,4}$ are anti-complete.

If $D_{5,6}\cup D_{6,1}\neq \emptyset$, then $G$ has a 4-coloring: 
\begin{align*}
&F_{6,1}\cup F_{5,6}\cup W\cup \{3\},\\
&F_{3,4}\cup D_{3,4}\cup \{6,1\},	\\
&D_{1,2}\cup D_{2,3}\cup \{4,5\},		\\
&D_{6,1}\cup D_{5,6}\cup \{2\}. 	\\	
\end{align*}

If $D_{5,6}\cup D_{6,1}= \emptyset$, then $G$ has a 4-coloring:
\begin{align*}
&F_{6,1}\cup F_{5,6}\cup W\cup \{3\},\\
&F_{3,4}\cup D_{3,4}\cup \{6\}\cup T_3,	\\
&D_{1,2}\cup D_{2,3}\cup \{4,5\}\cup T_2,		\\
&\{1,2\}\cup T_4. 	\\	
\end{align*}

Suppose now that $D_{3,4}$ and $F_{3,4}$ are not anti-complete and that $D_{5,6}$ and $F_{5,6}$ are anti-complete.
By \ref{itm:DT} and \ref{itm:TF}, $D_{3,4}\cup F_{3,4}$ are anti-complete to $T_3\cup T_4$.
Since $G$ is $2P_2$-free, it follows that $T_3$ and $T_4$ are anti-complete.
Then $G$ has a 4-coloring:
\begin{align*}
&F_{6,1}\cup F_{3,4}\cup W,\\
&F_{5,6}\cup D_{5,6}\cup \{2,3\},	\\
&D_{1,2}\cup D_{6,1}\cup \{4,5\}\cup T_2,		\\
&D_{2,3}\cup D_{3,4}\cup \{6,1\}\cup T_3\cup T_4. 	\\	
\end{align*}

$\bullet$ Suppose that $F_{5,6}=\emptyset$. 
If $F_{6,1}=\emptyset$, then $G$ has a 4-coloring as above.
So, we can assume that $F_{6,1}\neq \emptyset$. Let $f_{61}\in F_{6,1}$. 
If $d\in D_{2,3}$ and $f\in F_{2,3}$ are adjacent, then $\{2,f_{61},d,f\}$ induces a $K_4$ by \ref{itm:FiFj} and \ref{itm:DF}.
So, $D_{2,3}$ and $F_{2,3}$ are anti-complete.  
By \ref{itm:TF} and \ref{itm:TFextra}, $F_{6,1}$ and $T_2\cup T_3$ are complete.
Since $G$ is $K_4$-free, $T_2$ and $T_3$ are anti-complete.
Then $G$ has a 4-coloring:
\begin{align*}
&F_{3,4}\cup F_{6,1}\cup W,\\
&F_{2,3}\cup D_{1,2}\cup  D_{2,3}\cup \{5,6\}\cup T_2\cup T_3,	\\
&D_{3,4}\cup \{1,2\}\cup T_4,		\\
&D_{5,6}\cup D_{6,1}\cup \{3,4\}. 	\\	
\end{align*}

$\bullet$ Suppose that $F_{3,4}=\emptyset$. If $F_{2,3}=\emptyset$, then $G$ has a  4-coloring as above.
So, we can assume that $F_{2,3}\neq \emptyset$. 
Let $f_{23}\in F_{2,3}$. 
If $d\in D_{6,1}$ and $f\in F_{6,1}$ are adjacent, then $\{1,f_{23},d,f\}$ induces a $K_4$ by \ref{itm:FiFj} and \ref{itm:DF}.
So, $D_{6,1}$ and $F_{6,1}$ are anti-complete. 

If $D_{5,6}\cup D_{6,1}\neq \emptyset$, then $G$ has a 4-coloring:
\begin{align*}
&F_{5,6}\cup F_{2,3}\cup W,\\
&F_{6,1}\cup D_{1,2}\cup D_{6,1}\cup \{3,4\},	\\
&D_{5,6}\cup \{1,2\},		\\
&D_{3,4}\cup D_{2,3}\cup \{5,6\}. 	\\	
\end{align*}

If $D_{5,6}\cup D_{6,1}\neq \emptyset$, then $G$ has a 4-coloring:
\begin{align*}
&F_{5,6}\cup F_{6,1}\cup W\cup \{3\},\\
&F_{2,3}\cup D_{1,2}\cup \{6\}\cup T_3,	\\
&D_{2,3}\cup \{4,5\}\cup T_2,		\\
&D_{3,4}\cup \{1,2\}\cup T_4. 	\\	
\end{align*}

\case{3} The set $F_{1,2}=\emptyset$ but the set $F_{4,5}\neq \emptyset$.
By \autoref{cla:F61F12F23}, either $F_{3,4}=\emptyset$ or $F_{5,6}=\emptyset$.
 By \ref{itm:DiDj}, \ref{itm:DF}, \ref{itm:FiFj} and $K_4$-freeness of $G$,
either $D_{2,3}$ and $F_{2,3}$ are anti-complete or $D_{6,1}$ and $F_{6,1}$ are anti-complete.

$\bullet$ Suppose that $F_{5,6}=\emptyset$.

If $D_{6,1}$ and $F_{6,1}$ are anti-complete, then $G$ has a 4-coloring:
\begin{align*}
&F_{2,3}\cup F_{3,4}\cup W\cup \{6\}\cup T_3,\\
&F_{6,1}\cup D_{1,2}\cup  D_{6,1}\cup \{3,4\},\\
&F_{4,5}\cup D_{5,6}\cup \{1,2\}\cup T_4,		\\
&D_{2,3}\cup D_{3,4}\cup \{5\}\cup T_2. 	\\	
\end{align*}

Now assume that $D_{2,3}$ and $F_{2,3}$ are anti-complete.

If $D_{5,6}\cup D_{6,1}\neq \emptyset$, then $G$ has a 4-coloring:
\begin{align*}
&F_{3,4}\cup F_{6,1}\cup W,\\
&F_{2,3}\cup D_{1,2}\cup  D_{2,3}\cup \{5,6\},	\\
&F_{4,5}\cup D_{3,4}\cup \{1,2\},		\\
&D_{5,6}\cup D_{6,1}\cup \{3,4\}. 	\\	
\end{align*}

If $D_{5,6}\cup D_{6,1}= \emptyset$, then $G$ has a 4-coloring:
\begin{align*}
&F_{3,4}\cup W\cup \{6\}\cup T_3,\\
&F_{2,3}\cup D_{1,2}\cup  D_{2,3}\cup \{5\}\cup T_2,	\\
&F_{4,5}\cup D_{3,4}\cup \{1,2\}\cup T_4,		\\
&F_{6,1}\cup \{3,4\}. 	\\	
\end{align*}

$\bullet$ Suppose that $F_{3,4}=\emptyset$. 
Suppose first that $D_{2,3}$ and $F_{2,3}$ are anti-complete.

If $D_{5,6}\cup D_{6,1}\neq \emptyset$, then $G$ has a 4-coloring:
\begin{align*}
&F_{5,6}\cup F_{6,1}\cup W\cup \{3\},\\
&F_{2,3}\cup D_{1,2}\cup  D_{2,3}\cup \{5,6\},	\\
&F_{4,5}\cup D_{3,4}\cup \{1,2\},		\\
&D_{5,6}\cup D_{6,1}\cup \{4\}. 	\\	
\end{align*}

If $D_{5,6}\cup D_{6,1}= \emptyset$, then  $G$ has  a 4-coloring:
\begin{align*}
&F_{5,6}\cup F_{6,1}\cup W\cup \{3\},\\
&F_{2,3}\cup D_{1,2}\cup  D_{2,3}\cup \{6\}\cup T_3,	\\
&F_{4,5}\cup D_{3,4}\cup \{1,2\}\cup T_4,		\\
&\{4,5\}\cup T_2. 	\\	
\end{align*}

Now suppose that $D_{2,3}$ and $F_{2,3}$ are not anti-complete and that $D_{6,1}$ and $F_{6,1}$ are anti-complete.
Then $T_2$ and $T_3$ are anti-complete for otherwise an edge between $T_2$ and $T_3$
and an edge between $D_{2,3}$ and $F_{2,3}$ induce a $2P_2$ by \ref{itm:DT} and \ref{itm:TF}.
Then $G$ has a 4-coloring:
\begin{align*}
&F_{5,6}\cup F_{2,3}\cup W,\\
&F_{6,1}\cup D_{1,2}\cup  D_{6,1}\cup \{3,4\},	\\
&F_{4,5}\cup D_{5,6}\cup \{1,2\}\cup T_4,		\\
&D_{2,3}\cup D_{3,4}\cup \{5,6\}\cup T_2\cup T_3. 	\\	
\end{align*}

\case{4} The set $F_{4,5}=\emptyset$ but the set $F_{1,2}\neq \emptyset$.
By \autoref{cla:F61F12F23}, either $F_{2,3}=\emptyset$ or $F_{6,1}=\emptyset$.
By \ref{itm:DF} and \ref{itm:FiFj}, $F_{3,4}$ is complete to $D_{5,6}\cup F_{5,6}$.
So, if $F_{3,4}\neq \emptyset$, then $D_{5,6}$ and $F_{5,6}$ are anti-complete for otherwise
$G$ would contain a $K_4$. By symmetry, if $F_{5,6}\neq \emptyset$, then $D_{3,4}$ and $F_{3,4}$ are anti-complete.
Moreover, either $D_{3,4}$ and $F_{3,4}$ are anti-complete or $D_{5,6}$ and $F_{5,6}$ are anti-complete.
Similarly,  either $D_{2,3}$ and $F_{3,4}$ are anti-complete or $D_{6,1}$ and $F_{5,6}$ are anti-complete.

$\bullet$ Suppose that $F_{6,1}=\emptyset$.
If both $F_{3,4}$ and $F_{5,6}$ are not empty, then consider the following 4-coloring of $G-(D_{2,3}\cup D_{6,1})$:
\begin{align*}
I_1&=F_{2,3}\cup D_{1,2}\cup W\cup \{6\}\cup T_3,\\
I_2&=F_{3,4}\cup D_{3,4}\cup \{1\}\cup T_4,	\\
I_3&=F_{5,6}\cup D_{5,6}\cup \{2,3\},		\\
I_4&=F_{1,2}\cup \{4,5\}\cup T_2. 	\\	
\end{align*}
If $D_{2,3}$ and $F_{3,4}$ are anti-complete, then $G$ has a 4-coloring: $I_1$, $I_2\cup D_{2,3}$, $I_3$ and $I_4\cup D_{6,1}$.
If $D_{6,1}$ and $F_{5,6}$ are anti-complete, then $G$ has a 4-coloring: $I_1$, $I_2$, $I_3\cup D_{6,1}$ and $I_4\cup D_{2,3}$.
It reamains to consider the case where at least one of $F_{3,4}$ and $F_{5,6}$ is empty.

Suppose that $F_{5,6}=\emptyset$. Recall that no vertex in $D_{1,2}$ can have a neighbor in both $F_{1,2}$ and $F_{3,4}$.
Let $D'_{1,2}$ be the set of vertices in $D_{1,2}$ that are anti-complete to $F_{1,2}$ and $D''_{1,2}=D_{1,2}\setminus D'_{1,2}$.
Then $G$ has a 4-coloring:
\begin{align*}
&F_{1,2}\cup D'_{1,2}\cup \{4,5\}\cup T_2,\\
&F_{2,3}\cup F_{3,4}\cup D''_{1,2}\cup W\cup \{6\}\cup T_3,	\\
&D_{2,3}\cup D_{3,4}\cup \{1\}\cup T_4,		\\
&D_{5,6}\cup  D_{6,1}\cup\{2,3\}. 	\\	
\end{align*}

Suppose now that $F_{5,6}\neq \emptyset$ and $F_{3,4}=\emptyset$. 
Note that no vertex in $D_{1,2}$ can have a neighbor in both $F_{1,2}$ and $F_{5,6}$. 
Let $D'_{1,2}$ be the set of vertices in $D_{1,2}$ that are anti-complete to $F_{1,2}$ 
and $D''_{1,2}=D_{1,2}\setminus D'_{1,2}$. Moreover, recall that since $F_{5,6}\neq \emptyset$,
$T_3$ and $T_4$ are anti-complete. Then $G$ has a 4-coloring:
\begin{align*}
&F_{1,2}\cup D_{2,3}\cup D'_{1,2}\cup \{4,5\}\cup T_2,\\
&F_{2,3}\cup F_{5,6}\cup D''_{1,2}\cup W,\\
&D_{3,4}\cup \{6,1\}\cup T_3\cup T_4, 	\\
&D_{5,6}\cup D_{6,1}\cup \{2,3\}.\\	
\end{align*}

$\bullet$ Suppose that $F_{2,3}=\emptyset$.
If both $F_{3,4}$ and $F_{5,6}$ are not empty,
then consider the following 4-coloring of $G-(D_{2,3}\cup D_{6,1})$:
\begin{align*}
I_1&=F_{6,1}\cup D_{1,2}\cup W\cup \{3\},\\
I_2&=F_{5,6}\cup D_{5,6}\cup \{2\},	\\
I_3&=F_{3,4}\cup D_{3,4}\cup \{6,1\}\cup T_3\cup T_4,		\\
I_4&=F_{1,2}\cup \{4,5\}\cup T_2. 	\\	
\end{align*}
If $D_{2,3}$ and $F_{3,4}$ are anti-complete, then $G$ has a 4-coloring: $I_1$, $I_2$, $I_3\cup D_{2,3}$ and $I_4\cup D_{6,1}$.
If $D_{6,1}$ and $F_{5,6}$ are anti-complete, then $G$ has a 4-coloring: $I_1$, $I_2\cup D_{6,1}$, $I_3$ and $I_4\cup D_{2,3}$.
So, one of $F_{3,4}$ and $F_{5,6}$ is empty.

Suppose that $F_{5,6}\neq \emptyset$. So, $F_{3,4}=\emptyset$. Recall that no vertex in $D_{1,2}$ can have a neighbor in both $F_{1,2}$ and $F_{5,6}$.
Let $D'_{1,2}$ be the set of vertices in $D_{1,2}$ that are anti-complete to $F_{1,2}$ 
and $D''_{1,2}=D_{1,2}\setminus D'_{1,2}$. Moreover,  $T_3$ and $T_4$ are anti-complete.
Then $G$ has a 4-coloring:
\begin{align*}
&F_{1,2}\cup D'_{1,2}\cup \{4,5\}\cup T_2,\\
&F_{6,1}\cup F_{5,6}\cup D''_{1,2}\cup W\cup \{3\},\\
&D_{6,1}\cup D_{5,6}\cup\{2\}, 	\\
&D_{2,3}\cup D_{3,4}\cup \{6,1\}\cup T_3\cup T_4.\\	
\end{align*}

Suppose now that $F_{5,6}=\emptyset$. 
Recall that no vertex in $D_{1,2}$ can have a neighbor in both $F_{1,2}$ and $F_{3,4}$.
Let $D'_{1,2}$ be the set of vertices in $D_{1,2}$ that are anti-complete to $F_{1,2}$ 
and $D''_{1,2}=D_{1,2}\setminus D'_{1,2}$. 

If $D_{5,6}\cup D_{6,1}\neq \emptyset$, then $G$ has a 4-coloring:
\begin{align*}
&F_{1,2}\cup D_{6,1}\cup D'_{1,2}\cup \{4,5\},\\
&F_{6,1}\cup F_{3,4}\cup D''_{1,2}\cup W,\\
&D_{5,6}\cup\{2,3\}, 	\\
&D_{2,3}\cup D_{3,4}\cup \{6,1\}.\\	
\end{align*}

If $D_{5,6}\cup D_{6,1}= \emptyset$, then $G$ has a 4-coloring:
\begin{align*}
&F_{1,2}\cup D_{2,3}\cup D'_{1,2}\cup \{4,5\}\cup T_2,\\
&F_{3,4}\cup D''_{1,2}\cup W\cup \{6\}\cup T_3,\\
&F_{5,6}\cup \{3\}, 	\\
&D_{3,4}\cup \{1,2\}\cup T_4.\\	
\end{align*}

In each case we have found a 4-coloring of $G$. This completes our proof.
\end{proof}

\section{Eliminate $H_2$}\label{sec:F2}

In this section we show that our main theorem, \autoref{thm:main}, holds when $G$ is connected, has no pair of comparable vertices, does not contain $H_1$ as an induced subgraph, but contains $H_2$ as an induced subgraph.

\begin{lemma}\label{lem:F2}
Let $G$ be a connected $(2P_2,K_4, H_1)$-free graph with no pair of comparable vertices.
If $G$ contains an induced $H_2$, then $\chi(G)\le 4$.
\end{lemma}

\begin{proof}
Let $H=C\cup \{f\}$ be an induced $H_2$ where $C=12345$ induces
a $C_5$ and $f$ is adjacent to $1$, $2$, $3$ and $4$. We partition $V\setminus C$
into subsets of $Z$, $R_i$, $Y_i$, $F_i$ and $U$ as in \autoref{sec:pre}.
By the fact that $G$ is $H_1$-free and \ref{itm:FiFi+2}, it follows that $F_i=\emptyset$
for $i\neq 5$. Note that $f\in F_5$. We choose $H$ such that

\medskip
$\bullet$ $|U|$ is minimum, and

\medskip
$\bullet$ $|F_5|$ is minimum subject to the previous condition.

\begin{enumerate}

\emitem {$U$ is complete to $R_i$ for $1\le i\le 5$.}\label{itm:URi}

Suppose not. Let $u\in U$ be nonadjacent to $r_i\in R_i$ for some $i$.
Suppose first that $1\le i\le 4$. Note that $C'=C\setminus \{i\}\cup \{r_i\}$ induces
a $C_5$ and $H'=C'\cup \{u\}$ induces an $H_2$. Since $5\in C'$, 
it follows that $F_5\cap U'=\emptyset$ and $U'\subseteq U$. 
Moreover, $u\in U$ is not in $U'$ since $u$ is not adjacent to $r_i$.
This implies that $|U'|<|U|$, contradicting the choice of $H$.

Now suppose that $i=5$. Note that $C'=C\setminus \{5\}\cup \{r_5\}$ induces
a $C_5$ and $H'=C'\cup \{u\}$ induces an $H_2$. 
Note that $U'\subseteq F_5\cup U$ and $u\notin U'$ since $u$ is not adjacent to $r_i$. 
By the chocie of $H$, there exists a vertex $f'\in F_5$ such that $f'$ is adjacent to $r_5$.  By \ref{itm:UYF},
$u$ and $f$ are not adjacent. But then $fr_5$ and $5u$ indcue a $2P_2$. \e

\emitem {If $U\neq \emptyset$, then $R_i$ and $R_{i+2}$ are anti-complete.}\label{itm:URiRi+2}

Let $u\in U$. If $r_i\in R_i$ and $r_{i+2}\in R_{i+2}$ are not adjacent, then 
$\{r_i,r_{i+2}, i+1,u\}$ induces a $K_4$, since $u$ is adjacent to $r_i$ and $r_{i+2}$
by \ref{itm:URi}. \e
\end{enumerate}

Suppose first that $U\neq \emptyset$. By \ref{itm:URiRi+2}, $R_i$ and $R_{i+2}$ are anti-complete.
Recall that $Y_i$ and $Y_{i+2}$ are anti-complete by \ref{itm:YiYi+2}. By \ref{itm:Ri+1YiYi+2},
$R_1$ is anti-complete to $Y_5\cup Y_2$ and $R_4$ is anti-complete to $Y_5\cup Y_3$. 
By \ref{itm:FY}, $F_5$ is anti-complete to $Y_1\cup Y_4$.
By \ref{itm:RiYi+1Ri+1Yi}, either $Y_3$ and $R_2$ are anti-complete or $Y_2$ and $R_3$ are anti-complete.

If $Y_3$ and $R_2$ are anti-complete, then $G$ admits the following $4$-coloring:
\begin{align*}
& Y_1\cup Y_4\cup U\cup F_5 				&\ref{itm:YiYi+2} \ref{itm:UYF} \ref{itm:FY} \\
& Y_2\cup Y_5\cup R_1\cup \{1\} 		&\ref{itm:YiYi+2} \ref{itm:Ri+1YiYi+2}\\
&  Y_3\cup R_2\cup R_4\cup \{2,4\}	&\ref{itm:URiRi+2} \ref{itm:Ri+1YiYi+2}\\
& R_3\cup R_5\cup Z\cup \{3,5\} 		&\ref{itm:URiRi+2} \ref{itm:UYF}
\end{align*}

If $Y_2$ and $R_3$ are anti-complete, then $G$ admits the following $4$-coloring:
\begin{align*}
& Y_1\cup Y_4\cup U\cup F_5 				&\ref{itm:YiYi+2} \ref{itm:UYF} \ref{itm:FY} \\
& Y_3\cup Y_5\cup R_4\cup \{4\} 		&\ref{itm:YiYi+2} \ref{itm:Ri+1YiYi+2}\\
&  Y_2\cup R_1\cup R_3\cup \{1,3\}	&\ref{itm:URiRi+2} \ref{itm:Ri+1YiYi+2}\\
& R_2\cup R_5\cup Z\cup \{2,5\} 		&\ref{itm:URiRi+2} \ref{itm:UYF}
\end{align*}

This shows that if $U\neq \emptyset$, then $G$ has a 4-coloring.
Therefore, we can assume in the following that $U=\emptyset$.
\begin{enumerate}
\setcounter{enumi}{2}

\emitem {Each vertex in $R_2\cup R_3$ is either complete or anti-complete to $F_5$.}\label{itm:FR2R3}

Suppose not. Let $r\in R_2\cup R_3$ be adjacent to $f\in F_5$ and not adjacent to $f'\in F_5$.
By symmetry, we may assume that $r\in R_2$. Note that $C'=C\setminus \{2\}\cup \{r\}$
induces a $C_5$ and $H'=C'\cup \{f\}$ induces an $H_2$. Clearly, $f'\notin F'_5$.
By the choice of $H$, there exists a vertex $y\in Y$ such that $y\in F'_5$. This means that
$y$ is not adjacent to $5$ but adjacent to $1$, $3$, $4$ and $r_2$. This implies that
$y\in Y_1$. By \ref{itm:FY}, $f'$ and $y$ are not adjacent. But now $f'2$ and $yr_2$
induce a $2P_2$. \e
\end{enumerate}

By \ref{itm:FY}, \ref{itm:FR} and \ref{itm:FR2R3}, only vertices in $R_5\cup Z$ can distinguish
two vertices in $F_5$. By \ref{itm:ZR}, $R_5\cup Z$ is an independent set and so $(F_5,R_5\cup Z)$ is a $2P_2$-free bipartite graph. 
This implies that $F_5=\{f\}$  since any two vertices in $F$ are comparable. 
Let $R'_i=N(f)\cap R_i$ and $R''_i=R_i\setminus R'_i$ for $i=2,3,5$.
We now prove properties of $R'_i$ and $R''_i$.

\begin{enumerate}
\setcounter{enumi}{3}

\emitem {$R'_5$ is anti-complete to $R'_2\cup R'_3$.}\label{itm:R'iR'j}

Suppose that $r'_5\in R'_5$ and $r'_2\in R'_2$ are adjacent. Then $\{r'_5,r'_2,1,f\}$
induces a $K_4$. \e

\emitem {$R'_5$ is anti-complete to $Y_2\cup Y_3$.}\label{itm:R'5Y}

Suppose that $r'_5\in R'_5$ and $y_2\in Y_2$ are adjacent. By \ref{itm:FY},
$f$ and $y_2$ are adjacent. Then $\{r'_5,4,y_2,f\}
$  induces a $K_4$. \e

\emitem {$R'_2$ is anti-complete to $R_4$. By symmetry, $R'_3$ is anti-complete to $R_1$.}\label{itm:R'2R4}

Suppose that $r'_2\in R'_2$ and $r_4\in R_4$ are adjacent. By \ref{itm:FR},
$f$ and $r_4$ are adjacent. Then $\{r'_2,r_4,3,f\} $  induces a $K_4$. \e

\emitem {$R''_5$ is anti-complete to $R''_2\cup R''_3$.}\label{itm:R''iR''j}

Suppose that $r''_5\in R''_5$ and $r''_2\in R''_2$ are adjacent. Then 
$r_5''r''_2$ and $f2$ induce a $2P_2$. \e

\emitem {$Y_5$ is anti-complete to $R''_2\cup R''_3$.}\label{itm:Y5R''2R''3}

Suppose that $y_5\in Y_5$ and $r''_2\in R''_2$ are adjacent.  By \ref{itm:FY},
$f$ and $y$ are not adjacent. Then  $y_5r''_2$ and $f4$ induce a $2P_2$. \e

\emitem {$R''_5$ is anti-complete to $Y_1\cup Y_4$.}\label{itm:R''5Y1Y4}

Suppose that $r''_5\in R''_5$ and $y_4\in Y_4$ are adjacent. 
By \ref{itm:FY}, $f$ and $y_4$ are not adjacent.
Then  $r''_5y_4$ and $f2$ induce a $2P_2$. \e

\emitem {$R''_2$ is anti-complete to $Y_1$. By symmetry, $R''_3$ is anti-complete to $Y_4$.}\label{itm:R''2Y1}

Suppose that $r''_2\in R''_2$ and $y_1\in Y_1$ are adjacent. 
By \ref{itm:FY}, $f$ and $y_1$ are not adjacent.
Then  $r''_2y_1$ and $f2$ induce a $2P_2$. \e

\emitem {$R'_2$ is anti-complete to $Y_3$. By symmetry, $R'_3$ is anti-complete to $Y_2$.}\label{itm:R'2Y3}

Suppose that $r'_2\in R'_2$ and $y_3\in Y_3$ are adjacent. By \ref{itm:FR},
$f$ and $y_3$ are adjacent. Then $\{r'_2,y_3,3,f\} $  induces a $K_4$. \e

\emitem {$Y_5$ is  complete to $R'_2\cup R'_3$.}\label{itm:Y5R'2R'3}

Suppose that $y_5\in Y_5$ and $r'_2\in R'_2$ are not adjacent.  By \ref{itm:FY},
$f$ and $y_5$ are not adjacent. Then 
$fr'_2$ and $5y_5$ induce a $2P_2$. \e
\end{enumerate}

We now prove properties of $Z$.
\begin{enumerate}
\setcounter{enumi}{12}
\emitem {Any vertex in $Z$ is anti-complete to either $Y_2$ or $Y_3$.}\label{itm:ZY2Y3}

Suppose not. Then there exists a vertex $z\in Z$ that is adjacent to a vertex $y_i\in Y_i$
for $i=2,3$. By \ref{itm:FY}, $f$ is adjacent to $y_2$ and $y_3$. Moreover,
$y_2$ and $y_3$ are adjacent by \ref{itm:YiYi+1}. This implies that $f$ and $z$
are not adjacent for otherwise $\{f,z,y_i,y_{i+1}\}$ would induce a $K_4$.

We now show that $z$ is anti-complete to $Y_1\cup Y_4\cup Y_5$.
Suppose not. Let $z$ be adjacent to a vertex $y\in  Y_1\cup Y_4\cup Y_5$.
Note that there exists a vertex $i\in N_C(f)$ such that $i$ is not adjacent to $y$.
Moreover, $f$ and $y$ are not adjacent by \ref{itm:FY}. Then $zy$ and $if$
induce a $2P_2$. This shows that $z$ is anti-complete to $Y_1\cup Y_4\cup Y_5$.
Recall that $Z$ is anti-complete to $R_i$ for each $i$ by \ref{itm:ZR}.
Therefore, $N(z)\subseteq Y_2\cup Y_3\subseteq N(f)$, contradicting the assumption
that $G$ has no pair of comparable vertices. \e

\emitem {If $z\in Z$ is not adjacent to $y_i\in Y_i$, then $y_i$ is complete to $N(z)\setminus Y_i$.}\label{itm:ZYi}

It suffices to prove for $i=1$ by symmetry. Note that $N(z)\setminus Y_1=(N(z)\cap (Y_2\cup Y_5))\cup (N(z)\cap (Y_3\cup Y_4))$.
By \ref{itm:YiYi+1}, $y_1$ is complete to $N(z)\cap (Y_2\cup Y_5)$. It remains to show that $y_1$ is complete
to $N(z)\cap (Y_3\cup Y_4)$. Suppose not. Let $y\in N(z)\cap (Y_3\cup Y_4)$ be  nonadjacent to $y_1$.
By symmetry, we may assume that $y\in Y_3$. Then $zy$ and $y_14$ induce a $2P_2$.  \e

\emitem {If $z$ is anti-complete to $Y_i$ for some $i\in \{2,3\}$, then $Y_i=\emptyset$.}\label{itm:Y2Y3empty}

Suppose that $z$ is anti-complete to $Y_2$ and $Y_2$ contains a vertex $y_2$.
It follows from \ref{itm:ZYi} that $N(z)\subseteq N(y_2)$, contradicting the assumption
that $G$ contains no pair of comparable vertices. \e
\end{enumerate}

If $Y_5=\emptyset$, then $N(5)=\{1,4\}\cup R_1\cup R_4\cup Y_2\cup  Y_3\subseteq N(f)$
by \ref{itm:FY} and \ref{itm:FR}. This contradicts the assumption that $G$ contains no pair of comparable vertices.
So, we assume in the following that $Y_5$ contains a vertex $y_5$.
We claim now that either $R''_2$ or $R''_3$ is empty. Suppose not.
Let $r''_i\in R''_i$ for $i=2,3$. By \ref{itm:RiRi+1}, $r''_2$ and $r''_3$ are adjacent.
Moreover, $y_5$ is not adjacent to $r''_2$ and $r''_3$ by \ref{itm:Y5R''2R''3}.
Then $r''_2r''_3$ and $5y_5$ induce a $2P_2$.
This proves that either $R''_2$ or $R''_3$ is empty.
We consider two cases depending on whether $f$ has a neighbor in $R_5$.

\medskip
\noindent {\bf Case 1.} $R'_5=\emptyset$, i.e., $f$ has no neighbor in $R_5$.
Therefore, $R_5=R''_5$.  
Recall that either $R''_2$ or $R''_3$ is empty.
By symmetry, we may assyme that $R''_2=\emptyset$.
Then $R_2=R'_2$ and so $R_2$ and $R_4$ are anti-complete by \ref{itm:R'2R4}.
Let $Y'_2=\{y\in Y_2: y \textrm{ is anti-complete to } Y_5\}$ and $Y''_2=Y_2\setminus Y'_2$.
Note that each vertex in $Y''_2$ has a neighbor in $Y_5$ by the definition and so is anti-complete to $Y_4$ by \ref{itm:YiYi-2Yi+2}. 
Then the following is a $4$-coloring $\phi$ of $G-(R_3\cup Z)$:
\begin{align*}
I_1&= Y'_2\cup Y_5\cup R_1 \cup \{1\}	&	\ref{itm:Ri+1YiYi+2} \\
I_2&= Y''_2\cup Y_4\cup R_3\cup \{3\} 	&	\textrm{Definition of $Y''_2$}\\
I_3&=  R_2(=R'_2)\cup R_4\cup Y_3\cup \{2,4\}	&	\ref{itm:Ri+1YiYi+2}  \ref{itm:R'2Y3}\\
I_4&= Y_1\cup R_5(=R''_5)\cup \{f,5\}					& \ref{itm:FY}	\ref{itm:R''5Y1Y4} 
\end{align*}

We now extend $\phi$ to $R_3$ as follows. Since $R_3$ is an independent set by \ref{itm:ZR}, it suffices to 
explain how to extend $\phi$ to each vertex in $R_3$ independently. Let $r_3\in R_3$ be an arbitrary vertex.
Suppose first that $r_3\in R'_3$. By  \ref {itm:R'2R4} and \ref{itm:R'2Y3}, $r_3$ is anti-complete to
$R_1\cup Y_2$. By \ref{itm:RiYi+1Yi+2}, $r_3$ is anti-complete to either $Y_4$ or $Y_5$. Therefore,
we can add $r_3$ to either $I_1$ or $I_2$. Now suppose that $r_3\in R''_3$. By \ref{itm:R''iR''j}
and \ref{itm:R''2Y1}, $r_3$ is anti-complete to $Y_4\cup R_5$.
By \ref{itm:RiYi+1Yi+2}, $r_3$ is anti-complete to either $Y_1$ or $Y_2$.
Therefore, we can add $r_3$ to either $I_2$ or $I_4$.
This shows that $G-Z$ admits a $4$-coloring $\phi'=(I'_1,I'_2,I'_3,I'_4)$ with $I_i\subseteq I'_i$ for each $1\le i\le 4$.

We now obtain a 4-coloring of $G$ by either extending $\phi'$ to $Z$ or by finding another 4-coloring of $G$. 
If $Z$ is anti-complete to $Y_3$, then we can extend $\phi'$ by adding $Z$ to $I'_3$.
So, we assume that there is a vertex $z\in Z$ that is adjacent to a vertex  in $Y_3$.
It then follows from \ref{itm:ZY2Y3} and \ref{itm:Y2Y3empty} that $Y_2=\emptyset$.
If each vertex in $Z$ is anti-complete to one of $Y_3$, $Y_4$ and $Y_5$, then
we can extend $\phi'$ to $Z$ by adding each vertex in $Z$ to $I'_1$, $I'_2$ or $I'_3$
(since $Y_2=\emptyset$). Therefore, let $z\in Z$ be adjacent to $y_i\in Y_i$ for $i\in \{3,4,5\}$. 
We prove some additional properties using the existence of $y_3$, $y_4$ and $y_5$.
First of all, $R_1$ and $R_4$ are anti-complete. Suppose not. Let $r_1\in R_1$ and $r_4\in R_4$
be adjacent. By \ref{itm:Ri+1YiYi+2}, $y_5$ is not adjacent to $r_1$ and $r_4$.
Then $r_1r_4$ and $zy_5$ induce a $2P_2$. Secondly, $y_3$ and $y_5$ are not adjacent
for otherwise $\{y_3,y_4,y_5,z\}$ induces a $K_4$. Thirdly, $Y_1$ and $Y_4$ are anti-complete
to each other. Suppose not. Then $Y_1$ contains a vertex $y_1$ that is not anit-complete to $Y_4$.
By \ref{itm:YiYi-2Yi+2}, $y_1$ is anti-complete to $Y_3$. Then $fy_3$ and $y_1y_5$ induce a $2P_2$.
Now $G$ admits the following $4$-coloring:
\begin{align*}
&  Y_1\cup R''_5 (=R_5)\cup Y_4 \cup \{f,5\}	&	\ref{itm:R''5Y1Y4} \\
&  Y_3\cup R'_2 (=R_2)\cup \{2\} 					&	\ref{itm:R'2Y3}\\
& R_1\cup R_4\cup Y_5\cup \{1,4\}	&	\ref{itm:Ri+1YiYi+2} \\
& R_3\cup Z\cup \{3\}					& 	\ref{itm:ZR}	
\end{align*}

\medskip
\noindent {\bf Case 2.} $R'_5\neq \emptyset$. Let $r'_5\in R'_5$.
If $r_1\in R_1$ and $r_4\in R_4$ are adjacent, then $\{r_1,r_4,r'_5,f\}$
induces a $K_4$ by \ref{itm:RiRi+1} and \ref{itm:FR}.
So, $R_1$ and $R_4$ are anti-complete. We now consider two subcases.

\medskip
\noindent {\bf Case 2.1.} $R''_2$ and $Y_3$ are not anti-complete.
Let $r''_2\in R''_2$ and $y_3\in Y_3$ be adjacent. We claim first that
$Y_1$ and $Y_4$ are anti-complete. Suppose not. Let $y_1\in  Y_1$
and $y_4\in Y_4$ be adjacent. Then $y_3$ and $y_4$ are adjacent by \ref{itm:YiYi+1}.
By \ref{itm:YiYi-2Yi+2}, $y_1$ is not adjacent to $y_3$. Moreover, $y_1$ is not adjacent
to $r''_2$ by \ref{itm:R''2Y1}. But now $4y_1$ and $y_3r_2''$ induce a $2P_2$.
This shows that $Y_1$ and $Y_4$ are anti-complete.
Moreover, $Y_2$ and $R_3$ are anti-complete by \ref{itm:RiYi+1Ri+1Yi}.
Therefore, the following is a $4$-coloring $\phi$ of $G-(R''_2\cup Z)$.
\begin{align*}
I_1&= R_4\cup Y_5 \cup R_1 \cup \{1,4\}	&	\ref{itm:Ri+1YiYi+2}\\
I_2&= Y_1\cup R''_5\cup Y_4 \cup \{f,5\}					&	\ref{itm:R''5Y1Y4}\\
I_3&=  R_3\cup  Y_2\cup \{3\}	&						\ref{itm:RiYi+1Ri+1Yi}	 \\
I_4&= Y_3\cup R'_2\cup R'_5\cup \{2\}					& 	\ref{itm:R'iR'j} \ref{itm:R'5Y} \ref{itm:R'2Y3}
\end{align*}

We now explain how to extend $\phi$ to each vertex in $R''_2\cup Z$.
Since $R''_2\cup Z$ is an independent set by \ref{itm:ZR}, this will give a $4$-coloring of $G$.
By \ref{itm:ZY2Y3}, we can add each vertex in $Z$ to either $I_3$ or $I_4$.
Let $s''\in R''_2$ be an arbitrary vertex. Then $s''$ is anti-complete to $R''_5\cup  Y_1$ by
\ref{itm:R''iR''j} and \ref{itm:R''2Y1}. If $s''$ is not anti-complete to $Y_3$, 
then $s''$ is anti-complete to $Y_4$ by \ref{itm:RiYi+1Yi+2} and thus we can  add $s''$ to $I_2$.
Now $s''$ is anti-complete to $Y_3$. We claim that $s''$ is anti-complete to $R'_5$.
Suppose not. Then $s''$ is adjacent to some vertex $r'\in R'_5$. Note that $y_3$ is not
adjacent to $s''$ by our assumption. Moreover, $y_3$ is not adjacent to $r'$ by \ref{itm:R'5Y}.
Then $s''r'$ and $5y_3$ induce a $2P_2$.  This shows that $s''$ is anti-complete to $R'_5$
and thus we can add $s''$ to $I_4$.

\medskip
\noindent {\bf Case 2.2.} $R''_2$ and $Y_3$ are anti-complete. By symmetry,
$R''_3$ and $Y_2$ are anti-complete. It follows from \ref{itm:R'2Y3} that
$R_2$ and $Y_3$ are anti-complete and $R_3$ and $Y_2$ are anti-complete.
Recall that either $R''_2$ or $R''_3$ is empty.
By symmetry, we may assume that $R''_2=\emptyset$.
Then $R_2=R'_2$.
We now claim that $R'_3$ is anti-complete to $Y_4$. Suppose not. Let $r'_3\in R'_3$
be adjacent to $y_4\in Y_4$. By \ref{itm:Y5R'2R'3}, $r'_3$ is adjacent to $y_5$.
But this contradicts \ref{itm:RiYi+1Yi+2}. So, $R'_3$ is anti-complete to $Y_4$. 
By symmetry, $R'_2$ is anti-complete to $Y_1$. This together with \ref{itm:R''2Y1}
implies that $R_3$ and $R_2$ are anti-complete to $Y_4$ and $Y_1$, respectively.
Let $Y'_4=\{y\in Y_4: y \textrm{ is anti-complete to } Y_1\}$ and $Y''_4=Y_4\setminus Y'_4$.
Note that each vertex in $Y''_4$ has a neighbor in $Y_1$ and so is anti-complete to $Y_2$ by \ref{itm:YiYi-2Yi+2}.
Now $G-Z$ admits a $4$-coloring $\phi$:
\begin{align*}
I_1&= R_4\cup Y_5 \cup R_1 \cup \{1,4\}	&	\ref{itm:Ri+1YiYi+2}\\
I_2&= Y_1\cup R''_5\cup Y_4' \cup \{f,5\}					&	\ref{itm:R''5Y1Y4}\\
I_3&=  R_3\cup  Y_2\cup Y''_4\cup \{3\}	&						\ref{itm:RiYi+1Ri+1Yi}	 \\
I_4&= Y_3\cup R'_2\cup R'_5\cup \{2\}					& 	\ref{itm:R'iR'j} \ref{itm:R'5Y} \ref{itm:R'2Y3}
\end{align*}

We now explain how to obtain a $4$-coloring of $G$ based on $\phi$.
If $Z$ is anti-complete to $Y_3$, then we can add $Z$ to $I_4$.
So, assume that there exists a vertex in $Z$ that is adjacent to some vertex in $Y_3$.
It then follows from \ref{itm:ZY2Y3} and \ref{itm:Y2Y3empty} that $Y_2=\emptyset$.
If each vertex in $Z$ is anti-complete to one of $Y_3$, $Y''_4$ and $Y_5$, then
we can extend $\phi'$ to $Z$ by adding each vertex in $Z$ to $I_1$, $I_3$ or $I_4$
(since $Y_2=\emptyset$). Therefore, let $z\in Z$ be adjacent to $y_i\in Y_i$ for $i\in \{3,5\}$
and be adjacent to $y_4\in Y''_4$. 
Note that $y_3$ and $y_5$ are not adjacent for otherwise $\{y_3,y_4,y_5,z\}$ induces a $K_4$. 
We claim that $Y_1$ and $Y_4$ are anti-complete to each other. Suppose not. 
Then $Y_1$ contains a vertex $y_1$ that is not anit-complete to $Y_4$.
By \ref{itm:YiYi-2Yi+2}, $y_1$ is anti-complete to $Y_3$. Then $fy_3$ and $y_1y_5$ induce a $2P_2$.
Now $G$ admits the following $4$-coloring:
\begin{align*}
& R_4\cup Y_5 \cup R_1 \cup \{1,4\}	&	\ref{itm:Ri+1YiYi+2}\\
&  Y_1\cup R''_5\cup Y_4\cup \{f,5\}					&	\ref{itm:R''5Y1Y4}\\
& R_3\cup  Z\cup \{3\}	&						\ref{itm:ZR}	 \\
&  Y_3\cup R'_2\cup R'_5\cup \{2\}					& 	\ref{itm:R'iR'j} \ref{itm:R'5Y} \ref{itm:R'2Y3}
\end{align*}

This completes the proof.
\end{proof}

\section{Eliminate $W_5$ and $C_5$}\label{sec:W5C5} 

In this section we prove two lemmas. The fist one states that our main theorem, \autoref{thm:main}, holds when $G$ is connected, has no pair of comparable vertices, does not contain $H_1$ or $H_2$ as an induced subgraph, but contains the $5$-wheel as an induced subgraph. The second lemma then assumes that $G$ is $W_5$-free as well, but contains an induced $C_5$.

\begin{lemma}\label{lem:W5}
Let $G$ be a $(2P_2,K_4,  H_1, H_2)$-free graph with no pair of comparable vertices.
If $G$ contains an induced $W_5$, then $\chi(G)\le 4$.
\end{lemma}

\begin{proof}
Let $W=C\cup \{u\}$ be an induced $W_5$ such that $C=12345$ induces a $C_5$ in this order  and $u$ is complete to $C$. 
We partition $V\setminus C$ into subsets of $Z$, $R_i$, $Y_i$, $F_i$ and $U$ as in \autoref{sec:pre}.
Note that $u\in U$. Since $G$ is $H_2$-free,  it follows that $F_i=\emptyset$ for each $i$.
We prove the following properties.

\begin{enumerate}
\emitem {$U$ is complete to $R$.}\label{itm:UcomR}

If $u'\in U$ is not adjacent to $r_i\in R_i$, then $C\setminus \{i\}\cup \{r_i,u\}$
induces an $H_2$. This contradicts our assumption that $G$ is $H_2$-free. \e

\emitem {$R_i$ and $R_{i+2}$ are anti-complete.}\label{itm:RianticRi+2}

Suppose that $r_i\in R_i$ and $r_{i+2}\in R_{i+2}$ are not adjacent. 
By \ref{itm:UcomR}, $u$ is adjacent to both $r_i$ and $r_{i+2}$.
This implies that $\{r_i,r_{i+2},i+1,u\}$ induces a $K_4$. \e

\emitem {$R_i$ and $Y_{i+1}$ are anti-complere.}\label{itm:RiYi+1}

It suffices to prove for $i=1$. If $r_1\in R_1$ and $y_2\in Y_2$ are adjacent,
then $C\setminus \{1\}\cup \{r_1,y_2\}$ induces an $H_2$, a contradiction. \e

\emitem {$Y_i$ and $Y_{i+2}$ are anti-complete.}\label{itm:YianticYi+2}

Since $U\neq \emptyset$, \ref{itm:YianticYi+2} follows directly from \ref{itm:YiYi+2}. \e

\end{enumerate}

It follows from \ref{itm:RianticRi+2}--\ref{itm:YianticYi+2} and \ref{itm:ZR}--\ref{itm:UYF}
that $G$ admits the following $4$-coloring:
\begin{align*}
& R_1\cup R_3\cup Z \cup \{1,3\}	&	\ref{itm:RianticRi+2} \ref{itm:ZR} \\
& R_2\cup Y_3\cup R_4\cup \{2,4\} &	\ref{itm:RianticRi+2} \ref{itm:RiYi+1}\\
&  Y_1\cup R_5\cup Y_4\cup \{5\}	&	\ref{itm:RiYi+1} \ref{itm:YianticYi+2}\\
& Y_2\cup Y_5\cup U 					&	\ref{itm:YianticYi+2} \ref{itm:UYF}
\end{align*}
This completes our proof.
\end{proof}

\begin{lemma}\label{lem:C5}
Let $G$ be a connected $(2P_2,K_4, H_1,H_2,W_5)$-free graph with no pair of comparable vertices. 
If $G$ contains an induced $C_5$, then $\chi(G)\le 4$.
\end{lemma}

\begin{proof}
Let  $C=12345$ be an induced  $C_5$ in this order. We partition $V\setminus C$
into subsets of $Z$, $R_i$, $Y_i$, $F_i$ and $U$ as in \autoref{sec:pre}.
Since $G$ is $(H_2,W_5)$-free, both $U$ and $F_i$ are empty. It then follows from
\autoref{lem:partition} that 
$V(G)=C\cup Z\cup (\bigcup_{i=1}^{5}R_i)\cup (\bigcup_{i=1}^{5}Y_i)$.
We first prove the following properties of $R_i$ and $Z$.
\begin{enumerate}
\emitem {Each vertex in $R_i$ is anti-complete to either $R_{i-2}$ or $R_{i+2}$.}\label{itm:Ri-2Ri+2}

It suffices to prove for $i=4$. Suppose that $r_4\in R_4$ is adjacent to a vertex $r_i\in R_i$
for $i=1,2$. By \ref{itm:RiRi+1}, $r_1$ and $r_2$ are adjacent. This implies that $\{r_1,r_2,3,4,5,r_4\}$
induces a subgraph isomorphic to $H_2$. This contradicts the assumption that $G$ is $H_2$-free. \e

\emitem {$R_i$ and $Y_{i+1}$ are anti-complete.}\label{itm:RiYi+1antic}

It suffices to prove for $i=1$. If $r_1\in R_1$ and $y_2\in Y_2$ are adjacent,
then $C\setminus \{1\}\cup \{r_1,y_2\}$ induces an $H_2$. \e

\emitem {Each vertex in $Z$ cannot have a neighbor in each of $Y_i$ for $1\le i\le 5$.}\label{itm:ZY}

Suppose that $z\in Z$ has a neighbor $y_i\in Y_i$ for each $1\le i\le 5$.
By \ref{itm:YiYi+1}, $y_i$ and $y_{i+1}$ are adjacent. This implies that
$y_i$ and $y_{i+2}$ are not adjacent, for otherwise $\{y_i,y_{i+1},y_{i+2},z\}$
induces a $K_4$. But now $\{y_1,y_2,y_3,y_4,y_5,z\}$ induces a $W_5$. \e

\emitem {If $z\in Z$ has a neighbor in each of $Y_i$, $Y_{i+1}$, $Y_{i+2}$ and $Y_{i+3}$,
then $Y_{i+4}$ is anti-complete to $N(z)$.}\label{itm:ZYantic}

It suffices to prove for $i=1$. Let $y_i\in Y_i$ be a neighbor of $z$ for $1\le i\le 4$.
By \ref{itm:ZY}, $z$ is anti-complete to $Y_5$ and so $N(z)\subseteq Y_1\cup Y_2\cup Y_3\cup Y_4$
by \ref{itm:ZR}.  Let $y_5$ be an arbitrary vertex in $Y_5$.
By \ref{itm:YiYi+1}, $y_5$ is complete to $Y_1\cup Y_4$.
Therefore it remains to show that $y_5$ is complete to $N(z)\cap (Y_2\cup Y_3)$.
If $y_5$ is not adjacent to a vertex $y\in N(z)\cap (Y_2\cup Y_3)$, then either
$3y_5$ or $2y_5$ forms a $2P_2$ with $zy$ depending on whether $y\in Y_2$ or $y\in Y_3$.  \e

\emitem {If $Z$ contains a vertex that has a neighbor in $Y_i$, $Y_{i+1}$, $Y_{i+2}$ and $Y_{i+3}$,  then $Y_{i+4}=\emptyset$.}
\label{itm:Yempty}

Let $z\in Z$ have neighbor in $Y_i$ for $1\le i\le 4$. By \ref{itm:ZY},  $z$ is anti-complete to $Y_5$.
If $Y_5$ contains a vertex $y$, then $N(z)\subseteq N(y)$ by \ref{itm:ZYantic}. This contradicts the assumption
that $G$ contains no pair of comparable vertices. \e
\end{enumerate}

Let $Y'_4=\{y\in Y_4: y \textrm{ is anti-complete to } Y_1\}$ and $Y''_4=Y_4\setminus Y'_4$.
Note that each vertex in $Y''_4$ has a neighbor in $Y_1$ by the definition and so is anti-complete to $Y_2$ by \ref{itm:YiYi-2Yi+2}. 
Similarly,  let $R'_4=\{r\in R_4: r \textrm{ is anti-complete to } R_1\}$ and $R''_4=R_4\setminus R'_4$. 
By \ref{itm:Ri-2Ri+2}, $R''_4$ is anti-complete to $R_2$. 
We now consider the following two cases.

\medskip
\noindent {\bf Case 1.}  $Z$ contains a vertex that has a neighbor in four $Y_i$.
It then follows from \ref{itm:Yempty} that $Y_j=\emptyset$ for some $j$.
We may assume by symmetry that $j=5$. These facts and \ref{itm:RiYi+1antic}
imply that $G$ admits the following $4$-coloring:
\begin{align*}
 &  Y_1\cup R_5\cup Y'_4 \cup \{5\},	 \\
 & Y_2\cup R_3\cup Y''_4\cup \{3\}, \\
& R_1\cup Z\cup R'_4\cup \{1\},	\\
& R_2\cup Y_3\cup R''_4 \cup \{2,4\}.\\ 
\end{align*}

\noindent {\bf Case 2.}  Each vertex in $Z$ has a neighbor in at most three $Y_i$.
Note that $G-Z$ admits the following $4$-coloring $\phi$ by \ref{itm:RiYi+1antic}:
\begin{align*}
I_1&= Y_1\cup R_5\cup Y'_4 \cup \{5\},	 \\
I_2&= Y_2\cup R_3\cup Y''_4\cup \{3\}, \\
I_3&=  R_1\cup Y_5\cup R'_4\cup \{1\},	\\
I_4&= R_2\cup Y_3\cup R''_4 \cup \{2,4\}.\\ 
\end{align*}
We now explain how to extend $\phi$ to $Z$.
For this purpose we partition $Z$ into the following two subsets:
\begin{equation*} \label{eq1}
\begin{split}
Z_1 & = \{z\in Z: z \textrm{ is anti-complete to either }  Y_3  \textrm{ or }  Y_5\}, \\
Z_2 &=Z\setminus Z_1.
\end{split}
\end{equation*}

We first claim that each vertex in $Z_2$ has a neighbor in $Y_4$. Suppose not.
Let $z\in Z_2$ be a vertex such that $z$ is anti-complete to $Y_4$. 
Since $z$ has a neighbor in both $Y_3$ and $Y_5$, $z$ is anti-complete to either $Y_1$ or $Y_2$ by 
the assumption that each vertex in $z$ has a neighbor in at most three $Y_i$.
If $z$ is anti-complete to $Y_1$, then $N(z)\subseteq Y_2\cup Y_3\cup Y_5\subseteq N(5)$.
If $z$ is anti-complete to $Y_2$, then $N(z)\subseteq Y_1\cup Y_3\cup Y_5\subseteq N(3)$.
In either case, it contradicts the assumption that $G$ contains no pair of comparable veritces.
This proves the claim. Consequently, $Z_2$ is anti-complete to $Y_1\cup Y_2$.
We now claim that each vertex in $Z_2$ is anti-complete to either $Y'_4$ or $Y''_4$.
Suppose not. Let $z\in Z_2$ have a neighbor $y'_4\in Y_4$ and a neighbor $y''_4\in Y''_4$.
By the definition of $Y''_4$, it follows that $y''_4$ has a neighbor $y_1\in Y_1$. 
Then $3y_1$ and $y'_4z$ induce a $2P_2$ since $y'_4$ is not adjacent to $y_1$.  
Now we can extend $\phi$ to $Z$ by adding each vertex in $Z_1$ to $I_3$ or $I_4$ and 
by adding each vertex in $Z_2$ to $I_1$ or $I_2$.
\end{proof}

\paragraph*{Acknowledgments.}
Serge Gaspers is the recipient of an Australian Research Council (ARC) Future Fellowship (FT140100048)
and acknowledges support under the ARC's Discovery Projects funding scheme (DP150101134).


\end{document}